\documentclass[a4paper,11pt, twoside]{article}
\usepackage{amsmath,amsthm,esint}
\usepackage{amssymb,latexsym}
\usepackage{mathrsfs}
\usepackage{enumerate}
\usepackage{graphicx}
\usepackage[colorlinks,citecolor=red]{hyperref}
\usepackage[numbers,sort&compress]{natbib}
\usepackage{indentfirst}
\usepackage{mathtools}
\usepackage{paralist,bbding,pifont}
\newtheorem{theorem}{Theorem}[section]

\newtheorem{lemma}{Lemma}[section]
\newtheorem{proposition}{Proposition}[section]
\newtheorem{definition}{Definition}[section]
\newtheorem{remark}{Remark}[section]
\newtheorem{assumption}{Assumption}[section]

\theoremstyle{definition} \theoremstyle{remark}
\numberwithin{equation}{section}
\allowdisplaybreaks

\begin{document}

	\markboth{Y.Z. Yang and Y. Zhou}{stochastic partial differential equations with non-local operator}
	
	\date{}
	
	\baselineskip 0.22in
	
	\title{The time fractional stochastic partial differential equations with non-local operator on $\mathbb{R}^{d}$}
	
	\author{ Yong Zhen Yang$^1$, Yong Zhou$^{1,2}$\\[1.8mm]
		\footnotesize {Correspondence: yozhou@must.edu.mo}\\
		\footnotesize  {$^{1}$ Faculty of Mathematics and Computational Science, Xiangtan University}\\
		\footnotesize  {Hunan 411105, P.R. China}\\[1.5mm]
		\footnotesize {$^2$ Macao Centre for Mathematical Sciences, Macau University of Science and Technology}\\
		\footnotesize {Macau 999078, P.R. China}\\[1.5mm]
	}
	
	\maketitle

	\begin{abstract}
	This paper establishes a comprehensive well-posedness and regularity theory for time-fractional stochastic partial differential equations on $\mathbb{R}^d$ driven by mixed Wiener--L\'evy noises. The equations feature a Caputo time derivative $\partial_t^\alpha$ ($0<\alpha<1$) and a spatial nonlocal operator $\phi(\Delta)$ generated by a subordinate Brownian motion, leading to a doubly nonlocal structure.
	For the case $p \ge 2$, we prove the existence, uniqueness, and sharp Sobolev regularity of weak solutions in the scale of $\phi$-Sobolev spaces $\mathcal{H}_p^{\phi,\gamma+2}(T)$. Our approach combines harmonic analysis techniques (Fefferman--Stein theorem, Littlewood--Paley theory) with stochastic analysis to handle the combined Wiener and L\'evy noise terms. In the special case of cylindrical Wiener noise, a dimensional constraint $d < 2\kappa_0\bigl(2 - (2\sigma_2 - 2/p)_+/\alpha\bigr)$ is obtained.~For the low-regularity case $1 \le p \le 2$, where maximal function estimates fail, we construct unique local mild solutions in $L_p(\mathbb{R}^d)$ for equations driven by pure-jump L\'evy space-time white noise, using stochastic truncation and fixed-point arguments.
	The results unify and extend previous theories by simultaneously incorporating time-space nonlocality and jump-type randomness.
		
		
	\end{abstract}
	
	\section{Introduction}
	Fractional calculus has established itself as a fundamental mathematical framework for characterizing complex systems throughout various scientific disciplines. Distinct from conventional calculus, fractional operators intrinsically account for nonlocal interactions and memory effects, rendering them exceptionally appropriate for modeling hereditary characteristics in physical systems, anomalous transport mechanisms, and viscoelastic material behavior. For comprehensive mathematical foundations of these applications, consult \cite{Gorenflo,Y zhou2}.
	
This paper investigates the following stochastic partial differential equation with non-local operators (NLSPDE) on $\mathbb{R}^{d}$:
\begin{align}\label{TSFSPDE}
	\partial_t^\alpha w &= \phi(\Delta)w + g(w) 
	+ \sum_{k=1}^{\infty} \partial^{\sigma_{1}}_{t} \int_{0}^{t} h^{k}(w) \, dB^{k}_{s} \nonumber \\
	&\quad + \sum_{k=1}^{\infty} \partial^{\sigma_{2}}_{t} \int_{0}^{t} f^{k}(w) \, dZ^{k}_{s},~t>0, x \in \mathbb{R}^{d};~w(0,\cdot)= w_{0},~x \in \mathbb{R}^{d}.
\end{align}
as well as the NLSPDE driven by L\'{e}vy time-space white noise on $\mathbb{R}^{d}$:
\begin{align}\label{TSFSPDE1}
	\partial_t^\alpha w = \phi(\Delta)w + g(w) + \partial^{\sigma_{2}-1}_{t} \eta(w) \dot{\mathcal{Z}},\quad t>0,x\in\mathbb{R}^{d};
	w(0) = w_{0},~x\in\mathbb{R}^{d}.
\end{align}
Here, $\alpha \in (0,1)$, $\sigma_{1} < \alpha + 1/2$, $\sigma_{2} < \alpha + 1/p$, $\{B^{k}_{t}\}$ is a sequence of independent real-valued Wiener processes, and $\{Z^{k}_{t}\}$ is a sequence of independent $d_{1}$-dimensional real-valued L\'{e}vy processes.
	The function $\phi$ is a Bernstein function with $\phi(0^{+})=0$, mapping $(0,\infty)$ to $(0,\infty)$, that satisfies
	\[
	(-1)^{k}\phi^{(k+1)}(x) \geq 0, \quad x > 0, \quad k = 0, 1, 2, \ldots
	\]
	The operator $\phi(\Delta):=-\phi(-\Delta)$ represents the generator of rotationally invariant subordinate Brownian motion with characteristic exponent $\phi(|\xi|^{2})$, defined as
	\[
	\phi(\Delta)w(x) = \mathcal{F}^{-1}\left(-\phi(|\xi|^{2})\mathcal{F}w(\xi)\right)(x), \quad w \in \mathcal{S}(\mathbb{R}^{d}).
	\]
	The functions $g$, $h$, $f$, and $\eta$ are nonlinear functions that depend on $(t,x,\omega)$ and the unknown function $w$. Such stochastic partial differential equations (SPDEs) can be used to model stochastic effects of particles in a medium with thermal memory, or particles subject to adhesion and trapping mechanisms \cite{Chen 1}.
Throughout this paper, we typically suppress the dependence on $\omega\in\Omega$ when functions depend on $(t,x,\omega)$.

The study of fractional stochastic partial differential equations remains an active research area in fractional calculus. Krylov \cite{Krylov} pioneered the $L_{p}(p\geq 2)$ theory for classical SPDEs on $\mathbb{R}^{d}$ with zero initial conditions, that is
$
dw=\Delta w+g\,dW_{t}.
$
His analytical approach, based on controlling sharp maximal functions of $\nabla w$, established maximal regularity of solutions. This methodology has been subsequently extended to SPDEs with various spatial operators. Kim \cite{Kim1} first applied this analytical framework to classical SPDEs with $\phi(\Delta)$-type spatial operators. Chen \cite{Chen 1} investigated the $L_{2}$ theory for equations with both divergence and non-divergence form time fractional derivatives:
\begin{align}\label{feisanduxing}
	\partial^{\alpha}_{t}w&=\left(a^{ij}w_{x^{i}x^{j}}+b^{i}w_{x^{i}}+cw+f(w)\right)\notag\\
	&\qquad+\sum_{k=1}^{\infty}\partial^{\gamma}_{t}\int_{0}^{t}\left(\sigma^{ijk}w_{x^{i}x^{j}}+\mu^{ik}w_{x^{i}}
	+\nu^{k}w+g^{k}(w)\right)\,dW^{k}_{s},
\end{align}
and
\begin{align}\label{sanduxing}
	\partial^{\alpha}_{t}w&=\left(D_{i}(a^{ij}w_{x^{i}x^{j}}+b^{i}w_{x^{i}}+f^{i}(w))+cw+h(w)\right)\notag\\
	&\qquad+\sum_{k=1}^{\infty}\partial^{\gamma}_{t}\int_{0}^{t}\left(\sigma^{ijk}w_{x^{i}x^{j}}+\mu^{ik}w_{x^{i}}
	+\nu^{k}w+g^{k}(w)\right)\,dW^{k}_{s},
\end{align}
where $\{W^{k}_{t}\}$ denotes a sequence of independent one-dimensional Wiener processes.Building upon scaling properties of fractional heat equation solution operators, Kim \cite{Kim6} extended Krylov's analytical method to establish Sobolev regularity theory for solutions of \eqref{feisanduxing} and \eqref{sanduxing}. Subsequently, Kim \cite{Kim3} employed a combination of Krylov's techniques and $H^{\infty}$ calculus, along with fixed-point arguments, to develop Sobolev theory for time-fractional SPDEs driven by $\phi(\Delta)$-type operators:
\begin{align*}
	\partial^{\alpha}_{t}w&=\phi(\Delta)w+f(u)+\sum_{k=1}^{\infty}\partial^{\beta}_{t}\int_{0}^{t}g^{k}(w)\,dB_{s},\\
	\partial^{\alpha}_{t}w&=\phi(\Delta)w+f(u)+\partial^{\beta-1}_{t}g(w)\dot{W},
\end{align*}
with applications to Gaussian space-time white noise. Additional results concerning mild solutions of stochastic partial differential equations can be found in \cite{Wu} and references therein.
	
Equation \eqref{TSFSPDE} incorporates both temporal non-locality through the Caputo derivative $\partial^{\alpha}_{t}$ and spatial non-locality through the operator $\phi(\Delta)$. The Caputo derivative $\partial^{\alpha}_{t}$ effectively models subdiffusive behaviors arising from phenomena such as particle adhesion and trapping, as discussed in \cite{Gorenflo, Y zhou2}. The spatial non-local operator $\phi(\Delta)$, which serves as the infinitesimal generator of a subordinate Brownian motion, captures long-range particle jumps, diffusion on fractal structures, and the long-term behavior of particles moving in quenched disordered force fields; see \cite{Fogedby, Bouchaud}. Notably, when $\phi(x)=x^{\frac{\beta}{2}}$ with $0<\beta<2$, the operator $\phi(\Delta)$ reduces to the fractional Laplacian $(-\Delta)^{\frac{\beta}{2}}$, establishing a connection to isotropic $\beta$-stable processes. Furthermore, SPDE \eqref{TSFSPDE} can model stochastic effects in media with thermal memory or particle behaviors subject to adhesion and trapping mechanisms; see \cite{Chen 1}.
Moreover, it is important to note that since random effects in natural phenomena can be discontinuous in time, it is of significant practical relevance to consider stochastic partial differential equations driven simultaneously by both Wiener processes and L\'{e}vy processes. The model we consider, \eqref{TSFSPDE}, accommodates a non-zero initial value $w_{0}$, making it a natural generalization of the models studied by K.H. Kim and others \cite{Chen, Kim4, Kim3, Kim6}.

Our main contributions are as follows. For $p \geq 2$, by employing harmonic analysis techniques, we prove the existence, uniqueness, and regularity estimates for weak solutions in appropriate Sobolev spaces for the time-space fractional stochastic partial differential equation (TSFSPDE) \eqref{TSFSPDE}. Furthermore, we apply this regularity result to the nonlinear stochastic partial differential equation (NLSPDE) \eqref{TSFSPDE1} driven by a cylindrical Wiener process, under the dimensional constraint $d < 2\kappa_{0}\big(2 - \frac{(2\sigma_{2} - 2/p)_{+}}{\alpha}\big)$, thereby obtaining the regularity result for \eqref{TSFSPDE1}. To achieve this, noting the simultaneous presence of both the Wiener process and the L\'{e}vy process, we employ distinct techniques to handle the respective differences arising from each.
For the stochastic term induced by the Wiener process, our approach differs from that of Kim \cite{Kim3}, who controlled the sharp maximal function of the nonlocal derivative of the solution operator via the Hardy-Littlewood function of the free term $h$. Instead, we rely primarily on the Fefferman-Stein theorem and Marcinkiewicz interpolation to provide an alternative proof of the result in Kim \cite{Kim3}; our procedure is entirely different. Specifically, for the $B^{k}_{t}$ case, we transform estimates of $\phi(\Delta)^{\frac{\tilde{\delta}_{0}}{2}}w$ into bounds for the operator $\mathbb{S}h(t,x)$:
\[
\mathbb{S}h(t,x) = \left( \int_{-\infty}^{t} \left| \mathcal{S}^{\frac{\tilde{\delta}_{0}}{2}}_{\alpha,\sigma_{1}}(t-s) \star h \right|_{H}^{2} \, ds \right)^{\frac{1}{2}},
\]
which furnishes an alternative proof to Kim \cite{Kim3} in Lemma \ref{strong type p}.
For the stochastic term generated by the L\'{e}vy process, we rely on tools from stochastic analysis and harmonic analysis, combining the Burkholder-Davis-Gundy inequality with the Littlewood-Paley localization method to derive sharp upper bound estimates for the nonlocal derivative of the solution operator. Specifically, we derive the sharp upper bound estimates for $\left( \phi(-\Delta) \right)^{\frac{\tilde{\delta}_{1}+\varepsilon}{2}}w$, namely
\[
\big\| \left( \phi(-\Delta) \right)^{\frac{\tilde{\delta}_{1}+\varepsilon}{2}} w \big\|^{p}_{\mathscr{L}_{p}(T)}
\leq C \sum_{r=1}^{d_{1}} \left\| \int_{0}^{t} \left| \left( \phi(-\Delta) \right)^{\frac{\tilde{\delta}_{1}+\varepsilon}{2}} \mathcal{S}_{\alpha,\sigma_{2}}(t-s) \star f^{r}(s) \right|^{p} \, ds \right\|_{L_{1}\left( [0,T] \times \Omega; L_{1}(l_{2}) \right)}.
\]
By integrating the distinct estimates established above for the L\'{e}vy process and the Wiener process, and using methods from harmonic analysis along with a fixed-point theorem, we establish the regularity results for solutions to \eqref{TSFSPDE} and \eqref{TSFSPDE1} for $p \geq 2$. Moreover, we note that due to the limitations imposed by the sharp maximal function estimates for the derivative operators, the aforementioned regularity results fail for $1 \leq p < 2$. To address this, for $1 \leq p \leq 2$, inspired by \cite{Wu}, we employ techniques from stochastic analysis and a fixed-point theorem to establish the existence and uniqueness of local mild solutions to \eqref{TSFSPDE1}.
	
The remainder of this paper is organized as follows.~Section 2 presents fundamental concepts including fractional derivatives, Poisson random measures, and Bernstein operators $\phi(\Delta)$.~Section 3 develops crucial estimates through harmonic analysis methods and proves the existence and regularity of weak solutions in Sobolev spaces.~Section 4 establishes the existence and uniqueness of local mild solutions.

\section{Preliminaries}
	We introduce some necessary notions for this paper. We use $C$ to denote a generic constant that may change from line to line. We define the ball $B_\delta(x) := \{z \in \mathbb{R}^d : |x-z| < \delta\}$ with $B_\delta := B_\delta(0)$. For a multi-index $\gamma = (\gamma_1,\ldots,\gamma_d)$, we set
	\[
	\frac{\partial}{\partial_{x_{i}}}w=\nabla_{x_{i}}w, \quad \text{and} \quad \nabla^{\gamma}_{x}w=\nabla^{\gamma_{1}}_{x_{1}}\nabla^{\gamma_{2}}_{x_{2}}\dots\nabla^{\gamma_{d}}_{x_{d}}w,\quad |\gamma|=\gamma_{1}+\gamma_{2}+\cdots+\gamma_{d}.
	\]
	The space $L_{p}(X,\nu,B)$ consists of all $\nu$-measurable $B$-valued functions on $X$ such that
	\[
	\int_{X} \|w\|_{B}^{p}  d\nu < \infty,
	\]
	and we write $L_{p}(X,\nu,\mathbb{R})=L_{p}(X,\nu)$ for simplicity. Let $\mathcal{S}$ denote the Schwartz space, and $\mathcal{S}'$ its dual space, i.e., the space of tempered distributions. We use $\mathcal{F}:\mathcal{S}\rightarrow\mathcal{S}$ to denote the Fourier transform. Using duality, the Fourier transform $\mathcal{F}$ can be extended to $\mathcal{S}'$: for $u\in\mathcal{S}'$, we define $\mathcal{F}u$ by
	\[
	\langle\mathcal{F}u,\phi\rangle=\langle u,\mathcal{F}\phi\rangle
	\]
	for any $\phi\in\mathcal{S}$. We denote by $\mathcal{F}^{-1}$ the inverse Fourier transform. We use $*$ and $\star$ to denote convolution in time and space, respectively. If a function $f:\mathbb{R}_{+}\rightarrow\mathbb{R}$ is right-continuous with left limits, we say $f$ is c\`{a}dl\`{a}g.
	
	\begin{definition}
		For a function $f\in L_{1}\left(0,T;\mathcal{S}\right)$, the Caputo derivative $\partial^{\alpha}_{t}f$ for $0<\alpha<1$ is defined as
		\[
		\partial^{\alpha}_{t}f(t,x)=D^{\alpha}_{t}\left(f(t,x)-f(t,0)\right)=\frac{d}{dt}\left(g_{1-\alpha}\ast(f(t,x)-f(0,x))\right).
		\]
	\end{definition}
	
	The Mittag-Leffler function $E_{\alpha,\beta}(z)$ plays an important role in fractional calculus and is defined as
	\[
	E_{\alpha,\beta}(\varrho)=\sum_{k=0}^{\infty}\frac{\varrho^{k}}{\Gamma(k\alpha+\beta)},\quad \alpha,\beta>0,\varrho\in\mathbb{C}.
	\]
	The following relation can be found in \cite{Gorenflo}:
	\begin{align}\label{asympotic behavior}
		E_{\alpha,\beta}(\varrho)=\frac{1}{\pi\alpha}\int_{0}^{\infty}r^{\frac{1-\beta}{\alpha}}
		e^{-r^{\frac{1}{\alpha}}}\frac{r\sin(\pi(1-\beta))-\varrho\sin(\pi(1-\beta+\alpha))}{r^{2}-2r\varrho\cos(\pi\alpha)+\varrho^{2}}dr,
	\end{align}
	where $0<\alpha\leq 1$, $\beta<1+\alpha$, $|\arg(\varrho)|\geq \alpha\pi$, and $\varrho\neq 0$.
	
	Next, we introduce some facts about the operator $\phi(\Delta)$; for more details, see \cite{Kim1,Kim2,Kim3}. Let $\phi$ be a Bernstein function defined by
	\[
	\phi(x)=bx+\int_{(0,\infty)}\left(1-e^{-tx}\right)\nu(dt),\quad b\geq 0,\quad\int_{(0,\infty)}\left(1\wedge t\right)\nu(dt)<\infty.
	\]
	We easily obtain
	\begin{align}\label{B.S.T,func.bdd}
		|x^n \phi^{(n)}(x)| \leq b I_{n=1} + \int_{0}^{\infty} (t x)^n e^{-t x} \nu(dt) \lesssim \phi(x).
	\end{align}
	In this paper, we assume $b=0$, and we adopt the lower scaling condition from \cite{Kim1,Kim2}.
	
	\begin{assumption}
		There exist $\kappa_{0}\in(0,1]$ and $c_{1}>0$ such that
		\begin{align}\label{lower scailing condition}
			c_{1}\left(\frac{M}{m}\right)^{\kappa_{0}}\leq\frac{\phi(M)}{\phi(m)}\leq\frac{M}{m},\quad \text{for all } 0<m<M<\infty.
		\end{align}
	\end{assumption}
	
	From the above assumption, it is easy to obtain
	\begin{align}\label{Pro.Convergence}
		\int_{\varrho^{-1}}^{\infty}t^{-1}\phi(t^{-2})\,dt=\int_{1}^{\infty}t^{-1}
		\frac{\phi(\varrho^{2}t^{-2})}{\phi(\varrho^{2})}\phi(\varrho^{2})\,dt\leq C\int_{1}^{\infty}t^{-1-2\kappa_{0}}\,dt
		\phi(\varrho^{2})\leq C\phi(\varrho^{2}).
	\end{align}
	
	As is well known, for every Bernstein function $\phi$, there exists a subordinator $S_{t}$ defined on a probability space $\left(\Omega,\mathcal{F},\mathbb{P}\right)$ such that $\mathbb{E}[e^{-xS_{t}}]=e^{-t\phi(x)}$. Consider the $d$-dimensional subordinate Brownian motion $X_{t}:=W_{S_{t}}$, whose transition probability density function $p_{d}(t,x)$ can be expressed as
	\[
	p_{d}(t,x)=\int_{(0,\infty)}\frac{1}{(4\pi s)^{\frac{d}{2}}}\exp\left(\frac{-|x|^{2}}{4s}\right)\vartheta_{t}(ds),
	\]
	where $\vartheta_{t}$ is the distribution function of $S_{t}$. Therefore, $\phi(\Delta)$ is the infinitesimal generator of the subordinate Brownian motion $X_{t}$, i.e., for $g\in\mathcal{S}$,
	\[
	\phi(\Delta)g(x)=\lim_{t\rightarrow 0}\frac{\mathbb{E}g(x+X_{t})-g(x)}{t},
	\]
	which can equivalently be expressed as
	\[
	\phi(\Delta)g(x)=\mathcal{F}^{-1}\left(-\phi(|\xi|^{2})\mathcal{F}g(\xi)\right)(x),
	\]
	and
	\[
	\phi(\Delta)g(x)=\int_{\mathbb{R}^{d}}\left(g(x+y)-g(x)-\nabla g(x)y\textbf{I}_{|y|\leq 1}\right)j(|y|)\,dy,
	\]
	where the jump kernel $j$ is given by
	\[
	j(|y|)=\int_{(0,\infty)}(4\pi t)^{-\frac{d}{2}}\exp\left(-\frac{|y|^{2}}{4t}\right)\nu(dt).
	\]
	Moreover, for any $\zeta\in\left(0,1\right)$, the function $\phi^{\zeta}$ is also a Bernstein function, and
	\[
	\phi^{\zeta}(x)=\int_{(0,\infty)}\left(1-e^{-tx}\right)\nu_{\zeta}(dt),\quad\int_{(0,\infty)}\left(1\wedge t\right)\nu_{\zeta}(dt)<\infty.
	\]
	The operator $\phi^{\zeta}(\Delta)$ can be defined as
	\[
	\phi^{\zeta}(\Delta)g(x)=\mathcal{F}^{-1}\left(-\left(\phi(|\xi|^{2})\right)^{\zeta}\mathcal{F}g(\xi)\right)(x).
	\]
	Furthermore, we also have
	\[
	\phi^{\zeta}(\Delta)=\int_{\mathbb{R}^{d}}\left(g(x+y)-g(x)-\nabla g(x)y\textbf{I}_{|y|\leq r}\right)j_{\zeta}(|y|)\,dy,\quad \text{for any }r>0,
	\]
	where
	\[
	j_{\zeta}(|y|)=\int_{(0,\infty)}(4\pi t)^{-\frac{d}{2}}\exp\left(-\frac{|y|^{2}}{4t}\right)\nu_{\zeta}(dt),
	\]
	and $j_{\zeta}(|y|)\lesssim \left(\phi(|y|^{-2})\right)^{\zeta}/|y|^{d}$, see \cite{Kim3}.
	
	Let $X^{1}_{t}$ be a subordinator with characteristic exponent $\exp(-t\lambda^{\alpha})$, and $X^{2}_{t}$ be the inverse subordinator of $X^{1}_{t}$, i.e.,
	\[
	X^{2}_{t}=\inf\left\{s:X^{1}_{s}>t\right\}.
	\]
	Consider the subordinate process $Y_{t}:=W_{X^{2}_{t}}$, whose transition probability density $\mathcal{S}(t,x)$ is the fundamental solution to the following fractional equation:
	\[
	\partial^{\alpha}_{t}w(t,x)=\phi(\Delta)w(t,x),\quad w(0,\cdot)=\delta_{0},
	\]
	and
	\[
	\mathcal{S}(t,x)=\int_{0}^{\infty}p(t,x)\varpi(t,r)\,dr,
	\]
	where $\varpi(t,r)$ is the transition probability density of $X^{2}_{t}$. For any $\beta\in\mathbb{R}$, we denote $\varpi_{\alpha,\beta}(t,r)=D^{\beta-\alpha}_{t}\varpi(t,r)$ and define the functions
	\[
	\mathcal{S}_{\alpha,\beta}(t,x)=\int_{0}^{\infty}p(r,x)\varpi_{\alpha,\beta}(t,r)\,dr,\quad (t,x)\in (0,\infty)\times \mathbb{R}^{d}\setminus\{0\},
	\]
	and
	\[
	\mathcal{S}^{\zeta}_{\alpha,\beta}(t,x)=\int_{0}^{\infty}\phi(\Delta)^{\zeta}p(r,x)\varpi_{\alpha,\beta}(t,r)\,dr,\quad (t,x)\in (0,\infty)\times \mathbb{R}^{d}\setminus\{0\}.
	\]
	The following properties of $\mathcal{S}_{\alpha,\beta}(t,x)$ and $\mathcal{S}_{\alpha,\beta}^{\zeta}(t,x)$ can be found in \cite{Kim2,Kim3}.
	
	\begin{lemma}\label{Some proposition of S}
		For $k\in\mathbb{N}_{0}$, $\alpha\in\left(0,1\right)$, $\beta\in\mathbb{R}$, and $\zeta\in (0,1)$, we have the following facts:
		\begin{enumerate}[\rm(i)]
			\item For any $(t,x)\in (0,\infty)\times \mathbb{R}^{d}\setminus\{0\}$,
			\[
			\mathcal{S}_{\alpha,\beta}(t,x)=D^{\beta-\alpha}_{t}\mathcal{S}(t,x),
			\]
			\item $D^{k}_{x}\mathcal{S}_{\alpha,\beta}(t,x)$ and $D^{k}_{x}\mathcal{S}^{\zeta}_{\alpha,\beta}(t,x)$ are well-defined for $(t,x)\in (0,\infty)\times \mathbb{R}^{d}\setminus\{0\}$ and satisfy
			\begin{align}
				\left|D^{k}_{x}\mathcal{S}_{\alpha,\beta}(t,x)\right|&\lesssim t^{2\alpha-\beta}\frac{\phi(|x|^{-2})}{|x|^{d+k}},\label{daoshu estimate 1}\\
				\left|D^{k}_{x}\mathcal{S}^{\zeta}_{\alpha,\beta}(t,x)\right|&\lesssim t^{\alpha-\beta}\frac{\phi(|x|^{-2})^{\zeta}}{|x|^{d+k}},\label{daoshu estimate 2}
			\end{align}
			and for $t^{\alpha}\phi(|x|^{-2})\geq 1$,
			\begin{align}
				\left|D^{k}_{x}\mathcal{S}_{\alpha,\beta}(t,x)\right|&\lesssim\int_{\left(\phi(|x|^{-2})\right)^{-1}}^{2t^{\alpha}}
				\left(\phi^{-1}(\varrho^{-1})\right)^{\frac{d+k}{2}}t^{-\beta}\,d\varrho,\label{daoshu estimate 3}\\
				\left|D^{k}_{x}\mathcal{S}_{\alpha,\beta}^{\zeta}(t,x)\right|&\lesssim\int_{\left(\phi(|x|^{-2})\right)^{-1}}^{2t^{\alpha}}
				\left(\phi^{-1}(\varrho^{-1})\right)^{\frac{d+k}{2}}\varrho^{-\zeta}t^{-\beta}\,d\varrho.\label{daoshu estimate 4}
			\end{align}
			\item
			\begin{align}
				\int_{\mathbb{R}^{d}}|\mathcal{S}_{\alpha,\beta}(t,x)|\,dx\lesssim t^{\alpha-\beta},\quad \int_{\mathbb{R}^{d}}|\mathcal{S}^{\zeta}_{\alpha,\beta}(t,x)|\,dx\lesssim t^{\alpha(1-\zeta)-\beta},\label{L1 estimate}\\
				\mathcal{F}\mathcal{S}_{\alpha,\beta}(t,\xi)=t^{\alpha-\beta}E_{\alpha,1-\beta+\alpha}(-t^{\alpha}\phi(|\xi|^{2})),\label{Fourier transform 1}\\
				\mathcal{F}\mathcal{S}^{\zeta}_{\alpha,\beta}(t,\xi)=-t^{\alpha-\beta}\phi(|\xi|^{2})^{\zeta}
				E_{\alpha,1-\beta+\alpha}(-t^{\alpha}\phi(|\xi|^{2})).\label{Fourier transform 2}
			\end{align}
		\end{enumerate}
	\end{lemma}
	
	Next, we introduce some facts from stochastic analysis; see \cite{Ikeda,Wu,Wu2}. Let $\left(\Omega,\mathscr{F},\mathbb{P}\right)$ be a complete probability space, and $\mathscr{F}_{t}$ be a filtration of $\sigma$-algebras of $\mathscr{F}$ that is increasing and right-continuous. Let $\tilde{\mathscr{F}}$ be the $\sigma$-algebra generated by $\mathscr{F}_{t}$, i.e., $\tilde{\mathscr{F}}=\sigma\left\{\left(s,t\right]\times E:s<t,E\in\mathscr{F}_{s}\right\}$.
	
	For stochastic processes $X^{1}_{t},X^{2}_{t}$ with the same index set $t\in [0,T]$, we say $X^{2}_{t}$ is a modification of $X^{1}_{t}$ and write $X^{1}_{t}=X^{2}_{t}$ if
	\[
	\mathbb{P}\left\{\omega:X^{1}_{t}(\omega)=X^{2}_{t}(\omega),\forall t\in [0,T]\right\}=1.
	\]
	
To better understand Lévy noise and Gaussian noise, we begin by introducing the definition of the Poisson random measure~\cite{Ikeda}. For the measurable space $\left(A,\mathcal{B}(A),\mu(d\xi)\right)$, there exists a Poisson random measure $\Pi$ defined on
	\[
	\left([0,\infty)\times \mathbb{R}^{d}\otimes A,\mathcal{B}([0,\infty)\times\mathbb{R}^{d})\otimes\mathcal{B}(A),dtdx\otimes d\mu\right)
	\]
	such that
	\[
	\Pi:\mathcal{B}([0,\infty)\times\mathbb{R}^{d})\times\mathcal{B}(A)\times\Omega\rightarrow \mathbb{N}\cup\{0\}\cup\{\infty\},
	\]
	and
	\[
	\mathbb{E}\Pi\left([s,t],M,N,\cdot\right)=|t-s||M||N|,\quad \text{for any }[s,t]\times M\in\mathcal{B}(\mathbb{R}^{d}),N\in\mathcal{B}(A).
	\]
	In fact, such a Poisson random measure always exists; we can take $\Pi$ to be the canonical random measure
	\[
	\Pi([s,t]\times M\times N,\omega)=\sum_{n=1}^{\infty}\sum_{j=1}^{\delta_{n}(\omega)}\textbf{I}_{\{([s,t]\times M)\times A_{n})\times (N\times B_{n})\}}(\xi_{n,j}(\omega))\textbf{I}_{\{\omega:\delta_{n}(\omega)\geq 1\}}(\omega),
	\]
	with
	\[
	\mathbb{P}\{\omega\in\Omega:\xi_{n,j}(\omega)\in [s,t]\times M\times N\}=\frac{|t-s||M||N|}{|A_{n}||B_{n}|},
	\]
	for any $[s,t]\times M\in \mathcal{B}([s,t]\times \mathbb{R}^{d})\times A_{n}$ and $N\in \mathcal{B}(A)\times B_{n}$. Furthermore, we can take
	\[
	\mathscr{F}_{t}=\sigma\left\{\Pi([0,t]\times M\times N,\cdot):M\in\mathcal{B}(\mathbb{R}^{d}),N\in\mathcal{B}(A)\right\}
	\vee \mathcal{N},\quad \mathbb{P}(\mathcal{N})=0
	\]
	such that
	\[
	\left\{\Pi([0,t+s],M,N,\cdot)-\Pi([0,t],M,N,\cdot)\right\}_{s>0,(M,N)\in\mathcal{B}(\mathbb{R}^{d})\times\mathcal{B}(A)}
	\]
	is independent of $\mathcal{F}_{t}$. Based on the Poisson random measure $\Pi$, we can define the martingale measure $\tilde{\Pi}$ by
	\[
	\tilde{\Pi}(t,M,N,\omega)=\Pi([0,t]\times M,N,\omega)-t|M|\mu(N)
	\]
	with $\mathbb{E}[\tilde{\Pi}(t,M,N,\omega)]=0$ and $\mathbb{E}[|\tilde{\Pi}(t,M,N,\omega)|^{2}]=t|M|\mu(N)$. For an $\mathcal{F}_{t}$-predictable stochastic function $f$ satisfying
	\[
	\mathbb{E}\int_{0}^{t}\int_{M}\int_{N}|f(s,x,\xi)|\,ds\,dx\,\mu(d\xi)<\infty,
	\]
	we can define the $\mathcal{F}_{t}$-martingale
	\begin{align}\label{yangdefine}
		\int_{0}^{t}\int_{M}\int_{N}f(s,x,\xi,\omega)\tilde{\Pi}(ds dx d\xi,\omega)
		&:=\int_{0}^{t}\int_{M}\int_{N}f(s,x,\xi,\omega)\Pi(ds dx d\xi,\omega)\notag\\
		&\quad-\int_{0}^{t}\int_{M}\int_{N}f(s,x,\xi,\omega)\,ds\,dx\mu(d\xi).
	\end{align}
	Moreover, if
	\begin{align}\label{biancha}
		\mathbb{E}\int_{0}^{t}\int_{M}\int_{N}|f(s,x,\xi)|^{2}\,ds\,dx\,\mu(d\xi)<\infty,
	\end{align}
	then \eqref{yangdefine} is a square-integrable martingale with quadratic variation given by \eqref{biancha}.
	
	Note that from the definition of the martingale measure $\tilde{\Pi}$, by the Radon-Nikodym theorem, we can define
	\[
	\Pi_{t,x}(N,\omega)=\frac{\Pi(dtdx,N,\omega)}{dtdx}(t,x),\quad \tilde{\Pi}_{t,x}(N,\omega)
	=\frac{\tilde{\Pi}(dtdx,N,\omega)}{dtdx}(t,x)=\Pi_{t,x}(N,\omega)-\mu(N).
	\]
	By the L\'{e}vy-It\^{o} decomposition,
	\begin{align}\label{levyitodecomposition}
		Z_{t,x}(\omega)=W_{t,x}(\omega)+\int_{N_{0}}g_{1}(t,x,\xi,\omega)\,\mu(d\xi,\omega)+\int_{A\setminus N_{0}}g_{2}(t,x,\xi,\omega)\,\mu(d\xi,\omega),
	\end{align}
	where $W_{t,x}(\omega)$ is the Gaussian space-time white noise, $N_{0}\in \mathcal{B}(A)$, and $\mu(A\setminus N_{0})<\infty$. If $W_{t,x}=0$, we call $Z_{t,x}$ a pure jump L\'{e}vy space-time white noise.
\section{The Solvability And Sobolev Regularity}
In this section, we establish the solvability and Sobolev regularity with respect to NLSPDE \eqref{TSFSPDE} for $p\geq 2$.

For $p\geq 2$, consider the equation
\begin{align}\label{linear equation}
	\partial_t^\alpha w &= g(t,x)+\sum_{k=1}^{\infty}\partial_t^{\sigma_{1}}\int_{0}^{t}h^{k}(s,x)\,dB^{k}_{s}\notag\\
	&\quad+\sum_{k=1}^{\infty}\partial_t^{\sigma_{2}}\int_{0}^{t}f^{k}(s,x)\,dZ^{k}_{s},\quad t>0;\quad w(0) =\textbf{I}_{\alpha p>1}w_{0},
\end{align}
where $\sigma_{1}<\alpha+\frac{1}{2}$ and $\sigma_{2}<\alpha+\frac{1}{p}$.

We define the constants $0<\delta_{0},\delta<2$ by
\begin{align*}
	\delta_{0}&=\textbf{I}_{\sigma_{1}>\frac{1}{2}}(2\sigma_{1}-1)/\alpha+\kappa\textbf{I}_{\sigma_{1}=\frac{1}{2}},\\
	\delta_{1}&=\textbf{I}_{\sigma_{2}>\frac{1}{p}}(2\sigma_{2}-2/p)/\alpha+\kappa\textbf{I}_{\sigma_{2}=\frac{1}{p}},
\end{align*}
where $\kappa>0$ is small. Moreover, we define the initial space $U_{p}^{\phi,\gamma+2}$ as
\[
\mathbb{B}_{p,p}^{\phi,\gamma+2-\frac{2}{\alpha p}}=L_{p}\big(\Omega,\mathscr{F}_{0},B^{\phi,\gamma+2-\frac{2}{\alpha p}}_{p,p}\big).
\]
We define the following stochastic Banach spaces:
\[
\begin{aligned}
	\mathscr{H}^{\phi,\gamma}_{p}(T) &= L_{p}\big((0,T)\times\Omega,\tilde{\mathscr{F}},H_{p}^{\phi,\gamma}\big), &\mathscr{L}_{p}(T) &= \mathscr{H}^{\phi,0}_{p}(T), \\
	\mathscr{H}^{\phi,\gamma}_{p}(T,l_{2}) &= L_{p}\big((0,T)\times\Omega,\tilde{\mathscr{F}},H_{p}^{\phi,\gamma}(l_{2})\big), &\mathscr{L}_{p}(T,l_{2}) &= \mathscr{H}^{\phi,0}_{p}(T,l_{2}), \\
	\mathscr{H}^{\phi,\gamma}_{p}(T,l_{2},d_{1}) &= L_{p}\big((0,T)\times\Omega,\tilde{\mathscr{F}},H_{p}^{\phi,\gamma}(l_{2},d_{1})\big), &\mathscr{L}_{p}(T,l_{2},d_{1}) &= \mathscr{H}^{\phi,0}_{p}(T,l_{2},d_{1}),
\end{aligned}
\]
where $H^{\phi,\gamma}_{p}$ is the Sobolev space with respect to $\phi$, defined by
\[
\|u\|_{H^{\phi,\gamma}_{p}}=\|(I-\phi(\Delta))^{\frac{\gamma}{2}}u\|_{L_{p}},
\]
and the definitions of $H^{\phi,\gamma}_{p}(l_{2})$ and $H^{\phi,\gamma}_{p}(l_{2},d_{1})$ are similar.

Moreover, for any $k\in\mathbb{N}$, from the perspective of Poisson random measures, for the measurable space $(\mathbb{R}^{d_{1}},dy)$, there exists a Poisson random measure $\Pi^{k}$ defined on $([0,\infty)\otimes \mathbb{R}^{d_{1}},\mathcal{B}([0,\infty))\otimes\mathcal{B}(\mathbb{R}^{d_{1}}),dt\otimes dy)$. By the Radon-Nikodym theorem, for $N\in\mathcal{B}(\mathbb{R}^{d_{1}})$, we have
\[
\tilde{\Pi}^{k}(t,N,\omega)=\Pi^{k}(t,N,\omega)-t\mathbb{E}\Pi^{k}(1,N,\omega),\quad \Pi_{t}^{k}(N,\omega)=\frac{\Pi^{k}(dt,N,\omega)}{dt}(t),
\]
and $\tilde{\Pi}^{k}_{t}(N,\omega)=\Pi_{t}^{k}(N,\omega)-\mathbb{E}\Pi^{k}(1,N,\omega)$.~In fact, this is equivalent to
\begin{align*}
	\Pi_{t}^{k}(N,\omega) &:=\#\left\{0\leq s < t: \Delta Z^{k}_{s}:=Z^{k}_{s}-Z^{k}_{s^{-}}\in N\right\},\\
	\tilde{\Pi}^{k}(t,N,\omega) &:= \Pi^{k}(t,N,\omega) - t\mu^{k}(N), \quad \mu^{k}(N) = \mathbb{E}\Pi^{k}(1,N,\omega).
\end{align*}
Note that $Z_{t}^{k}$ is a $d_{1}$-dimensional L\'{e}vy process. Define
\[
(m_{p}(k))^{p} = \int_{\mathbb{R}^{d_{1}}} |y|^{p} \, \mu^{k}(dy), \quad \mu^{k}(N) = \mathbb{E}\tilde{\Pi}^{k}(1,N,\omega).
\]

Note that $Z_{t}^{k}$ is a $d_{1}$-dimensional L\'{e}vy process. Set
\[
(m_{p}(k))^{p}=\int_{\mathbb{R}^{d_{1}}}|y|^{p}\mu^{k}(dy),\quad \mu^{k}(N)=\mathbb{E}\tilde{\Pi}^{k}(1,N,\omega).
\]
If $m_{2}(k)<\infty$, then by the L\'{e}vy-It\^{o} decomposition, there exist a $d_{1}$-dimensional vector $a_{k}=(a^{1k},a^{2k},\ldots,a^{d_{1}k})$, a $d_{1}\times d_{1}$ matrix $b_{k}$, and a $d_{1}$-dimensional Brownian motion $\{\tilde{B}_{t}^{k}\}$ such that
\begin{align*}
	Z^{k}_{t}=a_{k}+b_{k}\tilde{B}_{t}^{k}+\int_{\mathbb{R}^{d_{1}}}y\,\tilde{\Pi}^{k}(t,dy),
\end{align*}
i.e.,
\begin{align*}
	Z^{ik}_{t}=a^{ik}+\sum_{j=1}^{d_{1}}b_{k}^{ij}\tilde{B}_{t}^{jk}+\int_{\mathbb{R}^{d_{1}}}y^{i}\,\tilde{\Pi}^{k}(t,dy),\quad i=1,2,\ldots,d_{1}.
\end{align*}
\begin{definition}\label{solution define}
	For $\gamma\in\mathbb{R}$, we say $w\in\mathcal{H}^{\phi,\gamma+2}_{p}(T)$ if there exist $w_{0}\in \mathbb{B}_{p,p}^{\phi,\gamma+2-\frac{2}{\alpha p}}$, $g\in \mathscr{H}^{\phi,\gamma+2}_{p}(T)$, $h\in\mathscr{H}^{\phi,\gamma+\delta_{0}}_{p}(T,l_{2})$, $f\in\mathscr{H}^{\phi,\gamma+\delta_{1}}_{p}(T,l_{2},d_{1})$ such that Equation \eqref{linear equation} holds in the distributional sense, i.e.,
	\begin{align}\label{solution in the sence of distribution}
		\langle w(t)-\textbf{I}_{\alpha p>1}w_{0},\varphi\rangle&=J_{t}^{\alpha}\langle g(t,\cdot),\varphi\rangle+\sum_{k=1}^{\infty}J_{t}^{\alpha-\delta_{1}}\int_{0}^{t}\langle h^{k}(s,\cdot),\varphi\rangle\,dB^{k}_{s}\notag\\
		&\quad+\sum_{k=1}^{\infty}J_{t}^{\alpha-\delta_{2}}\int_{0}^{t}\langle f^{k}(s,\cdot),\varphi\rangle\,dZ^{k}_{s}
	\end{align}
	holds almost everywhere on $\Omega\times [0,T]$.
\end{definition}

\begin{assumption}
	In this section, we assume the following conditions hold:
	\begin{enumerate}[\rm(i)]
		\item $M_{p}:=\sup_{k}m_{p}(k)<\infty$, for $p\geq 2$.
		\item $Z^{k}_{t}$ is a $d_{1}$-dimensional pure jump L\'{e}vy process, i.e., $a_{k}=0$, $b_{k}=0$.
	\end{enumerate}
\end{assumption}

\begin{remark}\label{some remark about yang}\text{ }
	
	\begin{enumerate}[\rm(i)]
		\item The condition $M_{p}<\infty$ is reasonable by \rm\cite[Remark 2.2]{Kim4}.
		
		\item If $m_{2}(k)<\infty$, then $Z^{ik}_{t}$ is a square-integrable martingale. For $f=\sum\limits_{j=1}^{m}a_{j}\textbf{I}_{(\tau_{j},\tau_{j+1}]}(t)$, where $\tau_{j}$ is a bounded stopping time, we can define the following square-integrable martingale $M_{t}^{k}$ with c\`{a}dl\`{a}g sample paths:
		$$
		M_{t}^{k}=\sum_{j=1}^{m}\int_{0}^{t}fdZ^{ik}_{t}=\sum_{j=1}^{m}a_{j}\left(Z^{ik}_{t\wedge\tau_{j+1}}-Z^{ik}_{t\wedge\tau_{j}}\right).
		$$
		Note that $\mathscr{H}^{\infty}_{0}(T,l_{2})$ is dense in $L_{2}\big([0,T],\tilde{\mathscr{F}},l_{2}\big)$. Therefore, for all $f\in L_{2}\big([0,T]\times\Omega,\tilde{\mathscr{F}},l_{2}\big)$, the stochastic integral $\int_{0}^{t}f\,dZ^{ik}_{t}$ becomes a square-integrable martingale with c\`{a}dl\`{a}g sample paths. Moreover, for $f=\left(f^{1},f^{2},...f^{d_{1}}\right)$, we have
		$$
		\int_{0}^{t}f\,dZ^{k}_{t}=\sum_{i=1}^{d_{1}}\int_{0}^{t}f^{i}\,dZ^{ik}_{t}=\sum_{i=1}^{d_{1}}\int_{0}^{t}\tilde{f}^{i}\,dZ^{ik}_{t},
		$$
		where $\tilde{f}=\big(\tilde{f}^{1},\tilde{f}^{2},...\tilde{f}^{d_{1}}\big)$ is the predictable projection of $f$.
		\item For $f\in L_{2}\left([0,T],\mathscr{F},\mathbb{R}^{d_{1}}\right)$,
		$$
		M_{t}^{k}=\int_{0}^{t}f\,dZ^{k}_{t}=\sum_{i=1}^{d_{1}}\int_{0}^{t}f^{i}\,dZ^{ik}_{t}=\sum_{i=1}^{d_{1}}\int_{0}^{t}
		\tilde{f}^{i}\,dZ^{ik}_{t}
		$$
		is a square-integrable martingale, whose quadratic variation is given by
		$$
		\langle M_{t}^{k},M_{t}^{k}\rangle=\sum_{i,j=1}^{d_{1}}\int_{0}^{t}y^{i}y^{j}f_{s}^{i}f_{s}^{j}\Pi(\,ds,\,dy),
		$$
		see \cite{Protter}.
		\item
		By the Burkholder-Davis-Gundy inequality and \cite[Lemma 2.5]{Chen}, there exists a constant $C=C(p,d_{1},T,m_{p})$ such that
		\begin{align*}
			\mathbb{E}\bigg[\sup_{s\leq t}\bigg|\sum_{k=1}^{\infty}M^{k}_{s}\bigg|^{p}\bigg]
			&\lesssim\mathbb{E}\bigg[\bigg(\sum_{k=1}^{\infty}\int_{0}^{T}\int_{\mathbb{R}^{d_{1}}}|y|^{2}|f^{k}(s)|^{2}
			\Pi_{x}(\,ds,\,dy)\bigg)^{\frac{p}{2}}\bigg]\\
			&\lesssim\mathbb{E}\bigg[\bigg(\int_{0}^{T}\sum_{k=1}^{\infty}|f^{k}(s)|^{2}\,ds\bigg)^{\frac{p}{2}}+\int_{0}^{T}
			\sum_{k=1}^{\infty}|f^{k}(s)|^{p}\,ds\bigg]\\
			&\lesssim\left\|f\right\|_{\mathscr{L}_{p}(T,l_{2},d_{1})}^{p}.
		\end{align*}
	\end{enumerate}
\end{remark}
	\begin{definition}\label{weak solution}
	We say $w$ is a weak solution (in the distributional sense) of Equation \eqref{TSFSPDE} if for any $\varphi\in\mathcal{S}$, the following holds almost everywhere on $\Omega\times [0,T]$:
	\begin{align*}
		\langle w(t)-w_{0},\varphi\rangle&=J_{t}^{\alpha}\langle\phi(\Delta)w,\varphi\rangle+J_{t}^{\alpha}\langle g(w),\varphi\rangle+\sum_{k=1}^{\infty}J_t^{\alpha-\sigma_{1}}\int_{0}^{t}\langle h(w)e_{k}(x),\varphi\rangle\,dB^{k}_{s}\\
		&\quad+\sum_{k=1}^{\infty}J_t^{\alpha-\sigma_{2}}\int_{0}^{t}\langle f(w)e_{k}(x),\varphi\rangle\,d\mathcal{Z}^{k}_{s}.
	\end{align*}
\end{definition}
We define the solution space $\mathcal{H}^{\phi,\gamma+2}_{p}(T)$ with norm
\begin{align*}
	\|w\|_{\mathcal{H}^{\phi,\gamma+2}_{p}(T)}=\|w\|_{\mathscr{H}^{\phi,\gamma+2}_{p}(T)}
	+\|w_{0}\|_{\mathbb{B}_{p,p}^{\phi,\gamma+2-\frac{2}{\alpha p}}}+\inf\|(g,h,f)\|_{\mathcal{F}^{\phi,\gamma+2}_{p}},
\end{align*}
where
\[
\mathcal{F}^{\phi,\gamma+2}_{p}=\mathscr{H}^{\phi,\gamma+2}_{p}(T)\times\mathscr{H}^{\phi,\gamma+\delta_{0}}_{p}(T,l_{2})\times
\mathscr{H}^{\phi,\gamma+\delta_{1}}_{p}(T,l_{2},d_{1}),
\]
and the infimum is taken over all $(g,h,f)\in\mathcal{F}^{\phi,\gamma+2}_{p}$ that satisfy Equation \eqref{linear equation} in the sense of Definition \ref{solution define}.
	\begin{lemma}
		For $\alpha\in (0,1)$, $k\in \mathbb{N}^{+}$, $i\in \{1,2,\ldots,d_{1}\}$, $f\in \mathscr{L}_{2}(T,l_{2})$, and for $X^{k}_{t}=B^{k}_{t}$ or $Z^{k}_{t}$, the following facts hold:
		\begin{enumerate}[\rm(i)]
			\item
			\begin{align*}
				J_{t}^{\alpha}\left(\sum_{k=1}^{\infty}\int_{0}^{\cdot}f^{k}(s)\,dX^{k}_{s}\right)(t)
				&=\sum_{k=1}^{\infty}\frac{1}{\Gamma(1+\alpha)}\int_{0}^{t}(t-s)^{\alpha-1}f^{k}(s)\,dX^{k}_{s},\quad \text{a.e. on }\Omega\times [0,T].
			\end{align*}
			\item
			\begin{align*}
				\partial_{t}^{\alpha}\left(\sum_{k=1}^{\infty}\int_{0}^{\cdot}f^{k}(s)\,dX^{k}_{s}\right)(t)
				&=\sum_{k=1}^{\infty}\frac{1}{\Gamma(1-\alpha)}\int_{0}^{t}(t-s)^{-\alpha}f^{k}(s)\,dX^{k}_{s},\quad \text{a.e. on }\Omega\times [0,T].
			\end{align*}
		\end{enumerate}
	\end{lemma}
	\begin{proof}
		The proof follows from \cite{Kim3,Kim5}.
	\end{proof}
	
	\begin{remark}\label{some usuful remark}
		The following facts hold; for detailed proofs, we refer to \cite{Kim2,Kim3,Kim4,Kim5,Kim6}.
		\begin{enumerate}[\rm(i)]
			\item The conditions $\sigma_{1}<\alpha+\frac{1}{2}$ and $\sigma_{2}<\alpha+\frac{1}{p}$ are necessary.
			\item The mapping $(I-\phi(\Delta))^{\frac{\nu}{2}}$ is an isometric isomorphism from $\mathcal{H}^{\phi,\gamma+2}_{p}(T)$ to $\mathcal{H}^{\phi,\gamma+2-\nu}_{p}(T)$.
			\item For $w\in\mathcal{H}^{\phi,\gamma+2}_{p}(T)$, $\Lambda\geq\max\{\alpha,\sigma_{1},\sigma_{2}\}$ and $\Lambda>\frac{1}{p}$, $J_{t}^{\Lambda-\alpha}w(t)$ has c\`{a}dl\`{a}g sample paths in $H^{\phi,\gamma}_{p}(T)$ and
			\begin{align*}
				&\langle J_{t}^{\Lambda-\alpha}(w(t)-\textbf{I}_{\alpha p>1}w_{0}),\varphi\rangle\\
				&=\langle J_{t}^{\Lambda-\alpha}g(t,\cdot),\varphi\rangle+\sum_{k=1}^{\infty}J_{t}^{\Lambda-\delta_{1}}\int_{0}^{t}\langle h^{k}(s,\cdot),\varphi\rangle\,dB^{k}_{s}\\
				&\quad+\sum_{k=1}^{\infty}J_{t}^{\Lambda-\delta_{2}}\int_{0}^{t}\langle f^{k}(s,\cdot),\varphi\rangle\,dZ^{k}_{s},
			\end{align*}
			and
			\begin{align*}
				\mathbb{E}\sup_{t\leq T}\|J^{\Lambda-\alpha}_{t}w\|_{H^{\phi,\gamma}_{p}}^{p}\leq C\left(\textbf{I}_{\alpha p>1}\mathbb{E}\|w_{0}\|_{H^{\phi,\gamma}_{p}}^{p}+\|g\|_{\mathscr{H}^{\phi,\gamma}_{p}(T)}^{p}+\|h\|_{\mathscr{H}^{\phi,\gamma}_{p}
					(T,l_{2})}^{p}+\|f\|^{p}_{\mathscr{H}^{\phi,\gamma}_{p}(T,l_{2},d_{1})}\right),
			\end{align*}
			where $C$ depends on $T,\alpha,\gamma,\delta_{0},d,d_{1},\Lambda,\sigma_{1},\sigma_{2}$.
			\item For $\theta=\min\{\alpha,1,2(\alpha-\sigma_{1})+1,p(\alpha-\sigma_{2})+2\}$, we have that for almost every $t\leq T$,
			\begin{align*}
				\|w\|_{\mathscr{H}^{\phi,\gamma}_{p}(t)}^{p}\leq C\int_{0}^{t}(t-s)^{\theta-1}\left(\|g\|_{\mathscr{H}^{\phi,\gamma}_{p}(s)}^{p}
				+\|f\|_{\mathscr{H}^{\phi,\gamma}_{p}(s,l_{2})}^{p}+\|h\|_{\mathscr{H}^{\phi,\gamma}_{p}(s,l_{2},d_{1})}^{p}\right)\,ds,
			\end{align*}
			where $C$ depends on $T,\alpha,\gamma,\delta_{0},d,d_{1},\Lambda,\sigma_{1},\sigma_{2}$.
		\end{enumerate}
	\end{remark}
	
	\subsection{Some estimates and Lemmas}
	\begin{lemma}\label{solution formula Lemma}
		For $\alpha\in(0,1)$, $\sigma_{2}<\alpha+\frac{1}{p}$, and $f\in \mathscr{H}^{\infty}_{0}(T,l_{2},d_{1})$,
		define the function
		\begin{align}\label{solution formula}
			w(t,x)=\sum_{k=1}^{\infty}\int_{0}^{t}\int_{\mathbb{R}^{d}}\mathcal{S}_{\alpha,\sigma_{2}}(t-s,x-y)f^{k}(s,y)\,dydZ^{k}_{s},
		\end{align}
		then $w\in\mathcal{H}^{\phi,2}_{p}(T)$ and satisfies the equation
		\begin{align}\label{model equation}
			\partial^{\alpha}_{t}w=\phi(\Delta)w+\sum_{k=1}^{\infty}\partial^{\sigma_{2}}_{t}\int_{0}^{t}f^{k}(s,x)\,dZ_{s}^{k},\quad w(0)=0.
		\end{align}
		Moreover, the results in \eqref{solution formula} and \eqref{model equation} also hold when $f$ and $Z^{k}_{s}$ are replaced by $g\in\mathscr{H}^{\infty}_{0}(T,l_{2})$ and $B^{k}_{t}$, respectively.
	\end{lemma}
	\begin{proof}
		The proof follows a similar approach to that in Kim \cite[Lemma 4.2]{Kim6} and Chen \cite[Lemma 3.10]{Chen 1}. The key distinctions lie in the application of \eqref{Fourier transform 1}, \eqref{Fourier transform 2}, and \cite[Lemma 4.1]{Kim2}, along with the substitution of the Wiener process $w_{t}^{k}$ by $Z^{k}_{t}$.
	\end{proof}
	We define the constant $\tilde{\delta}_{0},\tilde{\delta}_{1}>0$ that is
	$$
	\tilde{\delta}_{0}=2-(2\sigma_{1}-1)/\alpha \text{ and } \tilde{\delta}_{1}=2-(2\sigma_{2}-2/p)/\alpha.
	$$
	For any $(t_{0},x_{0})\in\mathbb{R}^{d+1}$ and constant $\varrho>0$, we denote
	\begin{align*}
		\lambda(\varrho)=\left(\phi(\varrho^{-2})\right)^{-\frac{1}{\alpha}},\text{ }B_{\varrho}(x_{0})=\left\{z:|x_{0}-z|<\varrho\right\},
	\end{align*}
	and
	\begin{align*}
		I_{\varrho}(t_{0})=\left(t_{0}-\lambda(\varrho),t_{0}\right),\text{ }\mathcal{Q}_{\varrho}(t_{0},x_{0})=I_{\varrho}(t_{0})\times B_{\varrho}(x_{0}),\text{ }\mathcal{Q}_{\varrho}:=\mathcal{Q}_{\varrho}(0,0).
	\end{align*}
	For the one dimension Brownian motion $B^{k}_{t}$, and define the solution
	$$
	w(t,x)=\sum_{k=1}^{\infty}\int_{0}^{t}\int_{\mathbb{R}^{d}}\mathcal{S}_{\alpha,\sigma_{1}}(t-s,x-y)h^{k}(s,y)\,dydB^{k}_{s},
	$$By using Burkholder-Davis-Gundy inequality, we derive that
	\begin{align*}
		\big\|\phi(\Delta)^{\frac{\tilde{\delta}_{0}}{2}}w\big\|_{\mathscr{L}_{p}(T)}
		&\lesssim\mathbb{E}\left\|\left(\int_{0}^{t}\sum_{k=1}^{\infty}
		\bigg|\int_{\mathbb{R}^{d}}
		\mathcal{S}^{\frac{\tilde{\delta}_{0}}{2}}_{\alpha,\sigma_{1}}(t-s,x-y)h^{k}(s,y)\,dy\bigg|^{2}\,ds\right)^{\frac{1}{2}}
		\right\|_{L_{p}((0,T)\times\mathbb{R}^{d})}\\
		&\lesssim\mathbb{E}\left\|\left(\int_{0}^{t}\big|\mathcal{S}_{\alpha,\sigma_{1}}^{\frac{\tilde{\delta}_{0}}{2}}(t-s)\star h\big|_{l_{2}}^{2}\,ds\right)^{\frac{1}{2}}\right\|_{L_{p}((0,T)\times\mathbb{R}^{d})}
	\end{align*}
	Denote $H=l_{2}$ and we can define the sublinear operator for $h\in C^{\infty}_{c}(\mathbb{R}^{d+1};H)$,
	\begin{align*}
		\mathbb{S}h(t,x)=\bigg(\int_{-\infty}^{t}\big|\mathcal{S}^{\frac{\tilde{\delta}_{0}}{2}}_{\alpha,\sigma_{1}}(t-s)\star h\big|_{H}^{2}\,ds\bigg)^{\frac{1}{2}}.
	\end{align*}
	\begin{lemma}\label{strong type p}
		For $p\geq 2$, $T\leq\infty$, $h\in L_{p}(\mathbb{R}^{d+1};H)$, we have
		\begin{align*}
			\int_{-\infty}^{T}\int_{\mathbb{R}^{d}}\left|\mathbb{S}h(t,x)\right|^{p}\,dxdt\lesssim \int_{-\infty}^{T}\int_{\mathbb{R}^{d}}\left|h(t,x)\right|_{H}^{p}\,dxdt,
		\end{align*}
		where the constant depends on $\alpha,\gamma,\delta_{0},d,\sigma_{1}$.
	\end{lemma}
	
	\begin{remark}
		Kim \cite{Kim3} proved Lemma \ref{strong type p} by controlling $\left|\mathbb{S}h(t,x)\right|^{2}$ via the Hardy-Littlewood maximal function $\mathcal{M}_{t}\mathcal{M}_{x}|h|^{2}_{H}(t,x)$. Here we provide an alternative proof using Marcinkiewicz interpolation and the Fefferman-Stein theorem.
	\end{remark}
	
	\begin{proof}
		Without loss of generality, we only verify the case $T=\infty$. Indeed, for $T<\infty$, we can take $\xi(t)\in C^{\infty}(\mathbb{R})$ such that $\xi(t)=1$ for $t\leq T$ and $\xi(t)=0$ for $t\geq T+\varepsilon$ for any $\varepsilon>0$. Then we replace $h$ by $\xi h$.
		
		The case $p=2$ follows from \cite[Lemma 3.5]{Kim3}. Therefore, we only need to prove the case $p>2$.
		First, note that $h\in C^{\infty}_{c}(\mathbb{R}^{d+1};H)$ is dense in $L_{p}(\mathbb{R}^{d+1};H)$, so we only consider $h\in C^{\infty}_{c}(\mathbb{R}^{d+1};H)$ and claim the following proposition:
		
		\begin{proposition}\label{dengjiaxingzhi}
			For any $(t_{0},x_{0})\in\mathbb{R}^{d+1}$, $\varrho>0$, and $(t,x)\in Q_{\varrho}(t_{0},x_{0})$, we have
			\begin{align}\label{BMO estimate}
				\fint_{Q_{\varrho}(t_{0},x_{0})}\left|\mathbb{S}h(t,x)-\big(\mathbb{S}h\big)_{Q_{\varrho}(t_{0},x_{0})}\right|\,dxdt\leq C \big\|h\big\|_{L_{\infty}(\mathbb{R}^{d+1};H)},
			\end{align}
			where the constant $C$ is independent of $T$.
		\end{proposition}
		
		\begin{proof}
			By change of variables, note that
			$$
			\fint_{Q_{\varrho}(t_{0},x_{0})}\left|\mathbb{S}h(t,x)-\big(\mathbb{S}h\big)_{Q_{\varrho}(t_{0},x_{0})}\right|\,dxdt
			=\fint_{Q_{\varrho}}\left|\mathbb{S}\tilde{h}(t,x)-\big(\mathbb{S}\tilde{h}\big)_{Q_{\varrho}}\right|\,dxdt,
			$$
			where $\tilde{h}(t,x)=h(t+t_{0},x+x_{0})$. Thus, without loss of generality, we only verify \eqref{BMO estimate} for $Q_{\varrho}$.
			We claim that for any $(t,x),(s,y)\in Q_{\varrho}$,
			\begin{align}\label{transform inequality}
				\fint_{Q_{\rho}}\fint_{Q_{\rho}}\big|\mathbb{S}h(t,x)-\mathbb{S}h(s,y)\big|\,dxdtdsdy\lesssim \big\|h\big\|_{L_{\infty}(\mathbb{R}^{d+1};H)}.
			\end{align}
			Let $\zeta\in C^{\infty}\left(\mathbb{R}^{d}\right)$ and $\eta\in C^{\infty}\left(\mathbb{R}\right)$ be cutoff functions satisfying $0\leq \zeta\leq 1$, $0\leq\eta\leq 1$ with
			\begin{align*}
				\zeta=
				\begin{cases}
					1 & \text{on } B_{\frac{7\varrho}{3}}\\
					0 & \text{on } B^{c}_{\frac{8\varrho}{3}}
				\end{cases},\quad
				\eta=
				\begin{cases}
					1 & \text{on } (\frac{-7\lambda(\varrho)}{3},\infty)\\
					0 & \text{on } (-\infty,\frac{-8\lambda(\varrho)}{3})
				\end{cases}.
			\end{align*}
			Thus, we have
			\begin{align*}
				\big|\mathbb{S}h(t,x)-\mathbb{S}h(s,y)\big|&\leq\big|\mathbb{S}h_{1}(t,x)-\mathbb{S}h_{1}(s,y)\big|
				+\big|\mathbb{S}h_{2}(t,x)-\mathbb{S}h_{2}(s,x)\big|\\
				&\quad +\big|\mathbb{S}h_{3}(s,x)-\mathbb{S}h_{3}(s,y)\big|
				+\big|\mathbb{S}h_{4}(s,x)-\mathbb{S}h_{4}(s,y)\big|,
			\end{align*}
			where
			$h_{1}=h\eta$ is supported in $\big(-3\lambda(\varrho),\infty\big)\times\mathbb{R}^{d}$, $h_{2}=h(1-\eta)$ is supported in $\big(-\infty,-2\lambda(\varrho)\big)\times\mathbb{R}^{d}$,
			$h_{3}=h(1-\eta)(1-\zeta)$ is supported in $\big(-\infty,-2\lambda(\varrho)\big)\times B^{c}_{2\varrho}$, and
			$h_{4}=h(1-\eta)\zeta$ is supported in $\big(-\infty,-2\lambda(\varrho)\big)\times B_{3\varrho}$.
			
			$\bullet$ Step 1: Estimate of $\fint_{Q_{\varrho}}\fint_{Q_{\varrho}}\big|\mathbb{S}h_{1}(t,x)-\mathbb{S}h_{1}(s,y)\big|\,dxdtdsdy$.
			
			Let $\xi\in C^{\infty}(\mathbb{R})$ such that $0\leq\xi\leq 1$, $\xi(t)=1$ on $|t|\leq 2\lambda(\varrho)$ and $\xi(t)=0$ for $|t|\geq5\lambda(\varrho)/2$. Note that $\mathbb{S}(h_{1}\xi)=\mathbb{S}h_{1}$ on $Q_{\varrho}$, and $|h_{1}\xi|\leq h_{1}$, we can assume $h_{1}(t,x)=0$ for $|t|\geq 3\lambda(\varrho)$. Moreover, let $\xi_{1}\in C^{\infty}(\mathbb{R}^{d})$ such that $\xi_{1}=1$ on $B_{5\varrho/2}$ and $\xi_{1}=0$ on $B^{c}_{7\varrho/3}$, hence we derive
			$$
			\fint_{Q_{\varrho}}|h_{1}(t,x)|\,dxdt\leq\fint_{Q_{\varrho}}|\mathbb{S}(h_{11})(t,x)|\,dxdt
			+\fint_{Q_{\varrho}}|\mathbb{S}(h_{12})(t,x)|\,dxdt,
			$$
			where $h_{11}=h_{1}\xi_{1}$ is supported in $(-3\lambda(\varrho),3\lambda(\varrho))\times B_{2\varrho}$, and $h_{12}=h_{1}(1-\xi_{1})$ is supported in $(-3\lambda(\varrho),3\lambda(\varrho))\times B^{c}_{2\varrho}$.
			Note that the operator $\mathbb{S}$ is strong type (2,2), and
			$$
			\int_{Q_{\varrho}}\big|\mathbb{S}h_{11}(t,x)\big|\,dxdt\leq|Q_{\varrho}|^{\frac{1}{2}}
			\bigg(\int_{Q_{\varrho}}\big|\mathbb{S}h_{11}(t,x)\big|^{2}\,dxdt\bigg)^{\frac{1}{2}}
			\leq|Q_{\varrho}|\|h\|_{L_{\infty}(\mathbb{R}^{d+1};H)}.
			$$
			Moreover, combining Lemma \ref{Some proposition of S} and noting that
			\begin{align*}
				&\bigg|\int_{\mathbb{R}^{d}}\mathcal{S}_{\alpha,\sigma_{1}}^{\frac{\tilde{\delta}_{0}}{2}}(t-\tau,z)h_{12}(\tau,x-z)\,dz\bigg|_{H}^{2}\\
				&\lesssim \big\|h\big\|^{2}_{L_{\infty}(\mathbb{R}^{d+1};H)}\textbf{I}_{|\tau|\leq 3\lambda(\varrho)}\bigg(\int_{|z|\geq \varrho}\big|\mathcal{S}_{\alpha,\sigma_{1}}^{\frac{\tilde{\delta}_{0}}{2}}(t-\tau,z)\big|
				\,dz\bigg)^{2}\\
				&\lesssim\big\|h\big\|^{2}_{L_{\infty}(\mathbb{R}^{d+1};H)}\textbf{I}_{|\tau|\leq 3\lambda(\varrho)}\bigg(\int_{\varrho}^{\infty}(t-\tau)^{\alpha-\sigma_{1}}
				\frac{(\phi(\kappa^{-2}))^{\frac{\tilde{\delta}_{0}}{2}}}{\kappa}\,d\kappa\bigg)^{2}\\
				&\lesssim\big\|h\big\|^{2}_{L_{\infty}(\mathbb{R}^{d+1};H)}\textbf{I}_{|\tau|\leq 3\lambda(\varrho)}(t-\tau)^{2(\alpha-\sigma_{1})}(\phi(\varrho^{-2}))^{\tilde{\delta}_{0}}.
			\end{align*}
			Thus, we derive
			\begin{align*}
				\big|\mathbb{S}h_{12}(t,x)\big|&=\bigg(\int_{-\infty}^{t}\big|\mathcal{S}^{\frac{\tilde{\delta}_{0}}{2}}_{\alpha,\sigma_{1}}(t-\tau)\star h\big|_{H}^{2}\,d\tau\bigg)^{\frac{1}{2}}\\&\lesssim \big\|h\big\|_{L_{\infty}(\mathbb{R}^{d+1};H)}
				(\phi(\varrho^{-2}))^{\frac{\tilde{\delta}_{0}}{2}}
				\bigg(\int_{|t-\tau|\leq 4\lambda(\varrho)}(t-\tau)^{2(\alpha-\sigma_{1})}\,d\tau\bigg)^{\frac{1}{2}}\lesssim\big\|h\big\|_{L_{\infty}(\mathbb{R}^{d+1};H)}.
			\end{align*}
			Thus, we obtain that
			$$
			\fint_{Q_{\varrho}}\fint_{Q_{\varrho}}\big|\mathbb{S}h_{1}(t,x)-\mathbb{S}h_{1}(s,y)\big|\,dxdtdsdy\lesssim \big\|h\big\|_{L_{\infty}(\mathbb{R}^{d+1};H)}.
			$$
			
			$\bullet$ Step 2: Estimate of $\fint_{Q_{\varrho}}\fint_{Q_{\varrho}}\big|\mathbb{S}h_{2}(t,x)-\mathbb{S}h_{2}(s,x)\big|\,dxdtdsdy$.
			
			Note that
			\begin{align*}
				&\big|\mathbb{S}h_{2}(t,x)-\mathbb{S}h_{2}(s,x)\big|\\&=
				\bigg|\bigg(\int_{-\infty}^{-2\lambda(\varrho)}\big|\mathcal{S}^{\frac{\tilde{\delta}_{0}}{2}}_{\alpha,\sigma_{1}}(t-\tau)\star h_{2}\big|_{H}^{2}\,d\tau\bigg)^{\frac{1}{2}}-
				\bigg(\int_{-\infty}^{-2\lambda(\varrho)}\big|\mathcal{S}^{\frac{\tilde{\delta}_{0}}{2}}_{\alpha,\sigma_{1}}(s-\tau)\star h_{2}\big|_{H}^{2}\,d\tau\bigg)^{\frac{1}{2}}\bigg|\\
				&\leq\bigg(\int_{-\infty}^{-2\lambda(\varrho)}\big|\big(\mathcal{S}^{\frac{\tilde{\delta}_{0}}{2}}_{\alpha,\sigma_{1}}(t-\tau)-
				\mathcal{S}^{\frac{\tilde{\delta}_{0}}{2}}_{\alpha,\sigma_{1}}(s-\tau)\big)\star h_{2}\big|_{H}^{2}\,d\tau\bigg)^{\frac{1}{2}},
			\end{align*}
			Note that $\alpha(1-\tilde{\delta}_{0}/2)-\sigma_{1}=-1/2$, we derive
			\begin{align*}
				\big|\big(\mathcal{S}^{\frac{\tilde{\delta}_{0}}{2}}_{\alpha,\sigma_{1}}(t-\tau)-
				\mathcal{S}^{\frac{\tilde{\delta}_{0}}{2}}_{\alpha,\sigma_{1}}(s-\tau)\big)\star h_{2}\big|_{H}^{2}&=
				\bigg|\int_{\mathbb{R}^{d}}\int_{s}^{t}
				\mathcal{S}^{\frac{\tilde{\delta}_{0}}{2}}_{\alpha,1+\sigma_{1}}(\theta-\tau,z)h_{2}(\tau,x-z)\,d\theta dz\bigg|_{H}^{2}\\
				&\lesssim\big\|h\big\|^{2}_{L_{\infty}(\mathbb{R}^{d+1};H)}\bigg(\int_{s}^{t}(\theta-\tau)^{-\frac{3}{2}}\,d\theta\bigg)^{2},
			\end{align*}
			By Minkowski's inequality,
			\begin{align*}
				\big|\mathbb{S}h_{2}(t,x)-\mathbb{S}h_{2}(s,x)\big|&\lesssim\big\|h\big\|_{L_{\infty}(\mathbb{R}^{d+1};H)}
				\int_{s}^{t}\bigg(\int_{-\infty}^{-2\lambda(\varrho)}\big(\theta-\tau\big)^{-3}\,d\tau\bigg)^{\frac{1}{2}}\,d\theta\lesssim
				\big\|h\big\|_{L_{\infty}(\mathbb{R}^{d+1};H)}.
			\end{align*}
			Thus, we obtain that
			$$
			\fint_{Q_{\varrho}}\fint_{Q_{\varrho}}\big|\mathbb{S}h_{2}(t,x)-\mathbb{S}h_{2}(s,x)\big|\,dxdtdsdy\lesssim \big\|h\big\|_{L_{\infty}(\mathbb{R}^{d+1};H)}.
			$$
			
			$\bullet$ Step 3: Estimate of $\fint_{Q_{\varrho}}\fint_{Q_{\varrho}}\big|\mathbb{S}h_{3}(s,x)-\mathbb{S}h_{3}(s,y)\big|\,dxdtdsdy$.
			
			Note that $h_{3}(\tau,z)=0$ for $\tau\geq -2\lambda(\varrho)$ or $|z|\leq 2\varrho$. Thus, by Minkowski's inequality, we derive
			\begin{align*}
				&\big|\mathbb{S}h_{3}(s,x)-\mathbb{S}h_{3}(s,y)\big|\\
				&\leq\bigg(\int_{-\infty}^{-2\lambda(\varrho)}\bigg|\int_{\mathbb{R}^{d}}
				\big(\mathcal{S}^{\frac{\tilde{\delta}_{0}}{2}}_{\alpha,\sigma_{1}}(s-\tau,x-z)-
				\mathcal{S}^{\frac{\tilde{\delta}_{0}}{2}}_{\alpha,\sigma_{1}}(s-\tau,y-z)\big)h_{3}(\tau,z)\,dz\bigg|_{H}^{2}\,d\tau\bigg)
				^{\frac{1}{2}}.
			\end{align*}
			Let $\theta(x,y,\mu)=\mu x+(1-\mu)y$ for $\mu\in (0,1)$ and combine Lemma \ref{Some proposition of S}
			\begin{align*}
				&\bigg|\int_{\mathbb{R}^{d}}
				\big(\mathcal{S}^{\frac{\tilde{\delta}_{0}}{2}}_{\alpha,\sigma_{1}}(s-\tau,x-z)-
				\mathcal{S}^{\frac{\tilde{\delta}_{0}}{2}}_{\alpha,\sigma_{1}}(s-\tau,y-z)\big)h_{3}(\tau,z)\,dz\bigg|_{H}\\
				&\leq\big\|h\big\|_{L_{\infty}(\mathbb{R}^{d+1};H)}\int_{|z|\geq 2\varrho}\int_{0}^{1}\big|\nabla\mathcal{S}^{\frac{\tilde{\delta}_{0}}{2}}_{\alpha,\sigma_{1}}(s-\tau,\theta(x,y,\mu)-z)
				\cdot(x-y)\big|\,d\mu\,dz\\
				&\lesssim\big\|h\big\|_{L_{\infty}(\mathbb{R}^{d+1};H)}\varrho\int_{|z|\geq \varrho}\big|\nabla\mathcal{S}^{\frac{\tilde{\delta}_{0}}{2}}_{\alpha,\sigma_{1}}(s-\tau,z)\big|
				\,dz
			\end{align*}
			Thus, we obtain
			\begin{align*}
				&\big|\mathbb{S}h_{3}(s,x)-\mathbb{S}h_{3}(s,y)\big|\\
				&\lesssim \varrho\big\|h\big\|_{L_{\infty}(\mathbb{R}^{d+1};H)}\bigg(\int_{-\infty}^{-2\lambda(\varrho)}\bigg(\int_{|z|\geq \varrho}\big|\nabla\mathcal{S}^{\frac{\tilde{\delta}_{0}}{2}}_{\alpha,\sigma_{1}}(s-\tau,z)\big|
				\,dz\bigg)^{2}\,d\tau\bigg)^{\frac{1}{2}}\\
				&\lesssim\varrho\big\|h\big\|_{L_{\infty}(\mathbb{R}^{d+1};H)}\bigg(\int_{\lambda(\varrho)}^{\infty}\bigg(\int_{|z|\geq \varrho}\big|\nabla\mathcal{S}^{\frac{\tilde{\delta}_{0}}{2}}_{\alpha,\sigma_{1}}(\tau,z)\big|
				\,dz\bigg)^{2}\,d\tau\bigg)^{\frac{1}{2}}\\
				&\lesssim\varrho\big\|h\big\|_{L_{\infty}(\mathbb{R}^{d+1};H)}\bigg[\bigg(\int_{\lambda(\varrho)}^{\infty}
				\bigg(\int_{|z|\geq(\phi^{-1}(\tau^{-\alpha}))^{-\frac{1}{2}}}
				\big|\nabla\mathcal{S}^{\frac{\tilde{\delta}_{0}}{2}}_{\alpha,\sigma_{1}}(\tau,z)\big|
				\,dz\bigg)^{2}\,d\tau\bigg)^{\frac{1}{2}}\\
				&\qquad+\bigg(\int_{\lambda(\varrho)}^{\infty}
				\bigg(\int_{\varrho\leq|z|\leq(\phi^{-1}(\tau^{-\alpha}))^{-\frac{1}{2}}}
				\big|\nabla\mathcal{S}^{\frac{\tilde{\delta}_{0}}{2}}_{\alpha,\sigma_{1}}(\tau,z)\big|
				\,dz\bigg)^{2}\,d\tau\bigg)^{\frac{1}{2}}\bigg]
			\end{align*}
			Note that $2\alpha-2\sigma_{1}-\alpha\tilde{\delta}_{0}=-1$ and combine Lemma \ref{Some proposition of S}, \eqref{Pro.Convergence} and \eqref{lower scailing condition}, we derive $\phi^{-1}(r^{-1})\leq \varrho^{-2}r^{-1}\phi(\varrho^{-2})^{-1}$, and
			\begin{align*}
				&\bigg(\int_{\lambda(\varrho)}^{\infty}\bigg(\int_{|z|\geq(\phi^{-1}(\tau^{-\alpha}))^{-\frac{1}{2}}}
				\big|\nabla\mathcal{S}^{\frac{\tilde{\delta}_{0}}{2}}_{\alpha,\sigma_{1}}(\tau,z)\big|
				\,dz\bigg)^{2}\,d\tau\bigg)^{\frac{1}{2}}\\
				&\lesssim\bigg(\int_{\lambda(\varrho)}^{\infty}\bigg(\int_{|z|\geq(\phi^{-1}(\tau^{-\alpha}))^{-\frac{1}{2}}}
				\big|\nabla\mathcal{S}^{\frac{\tilde{\delta}_{0}}{2}}_{\alpha,\sigma_{1}}(\tau,z)\big|
				\,dz\bigg)^{2}\,d\tau\bigg)^{\frac{1}{2}}\\
				&\lesssim\bigg(\int_{\lambda(\varrho)}^{\infty}\bigg(\int_{(\phi^{-1}(\tau^{-\alpha}))^{-\frac{1}{2}}}^{\infty}
				\frac{(\phi(\kappa^{-2}))^{\frac{\tilde{\delta}_{0}}{2}}}{\kappa^{2}}\,d\kappa
				\bigg)^{2}\tau^{2\alpha-2\sigma_{1}}\,d\tau\bigg)^{\frac{1}{2}}\\
				&\lesssim\bigg(\int_{\lambda(\varrho)}^{\infty}\phi^{-1}(r^{-\alpha})r^{-1}\,dr\bigg)^{\frac{1}{2}}\lesssim
				\bigg(\int_{\phi(\varrho^{-2})^{-1}}^{\infty}\phi^{-1}(r^{-1})r^{-1}\,dr\bigg)^{\frac{1}{2}}\lesssim \varrho^{-1},
			\end{align*}
			\begin{align*}
				&\bigg(\int_{\lambda(\varrho)}^{\infty}
				\bigg(\int_{\varrho\leq|z|\leq(\phi^{-1}(\tau^{-\alpha}))^{-\frac{1}{2}}}
				\big|\nabla\mathcal{S}^{\frac{\tilde{\delta}_{0}}{2}}_{\alpha,\sigma_{1}}(\tau,z)\big|
				\,dz\bigg)^{2}\,d\tau\bigg)^{\frac{1}{2}}\\
				&\lesssim\bigg(\int_{\lambda(\varrho)}^{\infty}\bigg(\int_{\varrho}^{(\phi^{-1}(\tau^{-\alpha}))^{-\frac{1}{2}}}
				\int_{(\phi(\kappa^{-2}))^{-1}}^{2\tau^{\alpha}}\big(\phi^{-1}(r^{-1})\big)^{\frac{d+1}{2}}
				r^{-\frac{\tilde{\delta}_{0}}{2}}\tau^{-\sigma_{1}}\kappa^{d-1}\,dr\,d\kappa\bigg)^{2}\,d\tau\bigg)^{\frac{1}{2}}\\
				&=\bigg(\int_{\lambda(\varrho)}^{\infty}\bigg(\int_{\phi(\varrho^{-2})^{-1}}^{2\tau^{\alpha}}
				\int_{\varrho}^{\big(\phi^{-1}(r^{-1})\big)^{-\frac{1}{2}}}\big(\phi^{-1}(r^{-1})\big)^{\frac{d+1}{2}}
				r^{-\frac{\tilde{\delta}_{0}}{2}}\tau^{-\sigma_{1}}\kappa^{d-1}\,d\kappa\,dr\bigg)^{2}\,d\tau\bigg)^{\frac{1}{2}}\\
				&\lesssim\bigg(\int_{\lambda(\varrho)}^{\infty}\bigg(\int_{\phi(\varrho^{-2})^{-1}}^{2\tau^{\alpha}}
				\big(\phi^{-1}(r^{-1})\big)^{\frac{1}{2}}
				r^{-\frac{\tilde{\delta}_{0}}{2}}\tau^{-\sigma_{1}}\,dr\bigg)^{2}\,d\tau\bigg)^{\frac{1}{2}}\\
				&\lesssim\int_{\phi(\varrho^{-2})^{-1}}^{\infty}
				\bigg(\int_{(\frac{r}{2})^{\frac{1}{\alpha}}}^{\infty}\tau^{-2\sigma_{1}}\,d\tau\bigg)^{\frac{1}{2}}
				\big(\phi^{-1}(r^{-1})\big)^{\frac{1}{2}}
				r^{-\frac{\tilde{\delta}_{0}}{2}}\,dr\\
				&\lesssim\int_{\phi(\varrho^{-2})^{-1}}^{\infty}
				\big(\phi^{-1}(r^{-1})\big)^{\frac{1}{2}}
				r^{-1}\,dr\lesssim \varrho^{-1}.
			\end{align*}
			Thus, we obtain that
			$$
			\fint_{Q_{\varrho}}\fint_{Q_{\varrho}}\big|\mathbb{S}h_{3}(s,x)-\mathbb{S}h_{3}(s,y)\big|\,dxdtdsdy\lesssim \big\|h\big\|_{L_{\infty}(\mathbb{R}^{d+1};H)}.
			$$
			
			$\bullet$ Step 4: Estimate of $\fint_{Q_{\varrho}}\fint_{Q_{\varrho}}\big|\mathbb{S}h_{4}(s,x)-\mathbb{S}h_{4}(s,y)\big|\,dxdtdsdy$.
			
			First, note $h_{4}$ is supported in $\big(-\infty,-2\lambda(\varrho)\big)\times B_{3\varrho}$, for any $(t,x)\in Q_{\varrho}$, we obtain
			\begin{align*}
				\big|\mathbb{S}h_{4}(t,x)\big|&\lesssim\bigg(\int_{-\infty}^{-2\lambda(\varrho)}
				\bigg|\int_{B_{3\varrho}}\mathcal{S}^{\frac{\tilde{\delta}_{0}}{2}}_{\alpha,\sigma_{1}}
				(t-\tau,x-z)h_{4}(\tau,z)\,dz\bigg|^{2}_{H}\,d\tau\bigg)^{\frac{1}{2}}\\
				&\lesssim\big\|h\big\|_{L_{\infty}(\mathbb{R}^{d+1};H)}\bigg(\int_{\lambda(\varrho)}^{\infty}
				\bigg(\int_{B_{4\varrho}}\big|\mathcal{S}^{\frac{\tilde{\delta}_{0}}{2}}_{\alpha,\sigma_{1}}
				(\tau,z)\big|\,dz\bigg)^{2}\,d\tau\bigg)^{\frac{1}{2}}
			\end{align*}
			Note that
			\begin{align*}
				&\int_{\lambda(\varrho)}^{\infty}
				\bigg(\int_{B_{4\varrho}}\big|\mathcal{S}^{\frac{\tilde{\delta}_{0}}{2}}_{\alpha,\sigma_{1}}
				(\tau,z)\big|\,dz\bigg)^{2}\,d\tau\\&\leq\int_{\lambda(\varrho)}^{\lambda(4\varrho)}
				\bigg(\int_{B_{4\varrho}}\big|\mathcal{S}^{\frac{\tilde{\delta}_{0}}{2}}_{\alpha,\sigma_{1}}
				(\tau,z)\big|\,dz\bigg)^{2}\,d\tau+\int_{\lambda(4\varrho)}^{\infty}
				\bigg(\int_{B_{4\varrho}}\big|\mathcal{S}^{\frac{\tilde{\delta}_{0}}{2}}_{\alpha,\sigma_{1}}
				(\tau,z)\big|\,dz\bigg)^{2}\,d\tau.
			\end{align*}
			By using Lemma \ref{Some proposition of S}, we derive
			\begin{align}\label{Min 1}
				\int_{\lambda(\varrho)}^{\lambda(4\varrho)}
				\bigg(\int_{B_{4\varrho}}\big|\mathcal{S}^{\frac{\tilde{\delta}_{0}}{2}}_{\alpha,\sigma_{1}}
				(\tau,z)\big|\,dz\bigg)^{2}\,d\tau\leq\int_{\lambda(\varrho)}^{\lambda(4\varrho)}\tau^{-1}\,d\tau\lesssim 1,
			\end{align}
			and
			\begin{align*}
				&\int_{\lambda(4\varrho)}^{\infty}
				\bigg(\int_{B_{4\varrho}}\big|\mathcal{S}^{\frac{\tilde{\delta}_{0}}{2}}_{\alpha,\sigma_{1}}
				(\tau,z)\big|\,dz\bigg)^{2}\,d\tau\\
				&\lesssim\int_{\lambda(4\varrho)}^{\infty}
				\bigg(\int_{B_{4\varrho}}\int_{(\phi(|z|^{-2}))^{-1}}^{2\tau^{\alpha}}
				\big(\phi^{-1}(r^{-1})\big)^{\frac{d}{2}}r^{-\frac{\tilde{\delta}_{0}}{2}}\tau^{-\sigma_{1}}\,dr\,dz\bigg)^{2}\,d\tau\\
				&\lesssim\int_{\lambda(4\varrho)}^{\infty}
				\bigg(\int_{B_{4\varrho}}\bigg[\int_{(\phi(|z|^{-2}))^{-1}}^{(\phi(\varrho^{-2}/16))^{-1}}
				+\int_{(\phi(\varrho^{-2}/16))^{-1}}^{2\tau^{\alpha}}\bigg]
				\big(\phi^{-1}(r^{-1})\big)^{\frac{d}{2}}r^{-\frac{\tilde{\delta}_{0}}{2}}\tau^{-\sigma_{1}}\,dr\,dz\bigg)^{2}\,d\tau.
			\end{align*}
			Note that $\tilde{\delta}_{0}-2+(2\sigma_{1}-1)/\alpha=0$, we derive
			\begin{align}\label{Min 2}
				&\int_{\lambda(4\varrho)}^{\infty}\bigg(\int_{B_{4\varrho}}\int_{(\phi(|z|^{-2}))^{-1}}^{(\phi(\varrho^{-2}/16))^{-1}}
				\big(\phi^{-1}(r^{-1})\big)^{\frac{d}{2}}r^{-\frac{\tilde{\delta}_{0}}{2}}\tau^{-\sigma_{1}}\,dr\,dz\bigg)^{2}\,d\tau\notag\\
				&\lesssim\int_{\lambda(4\varrho)}^{\infty}\bigg(\int_{0}^{(\phi(\varrho^{-2}/16))^{-1}}
				\int_{|z|\leq (\phi^{-1}(r^{-1}))^{-\frac{1}{2}}}
				\big(\phi^{-1}(r^{-1})\big)^{\frac{d}{2}}r^{-\frac{\tilde{\delta}_{0}}{2}}\tau^{-\sigma_{1}}\,dz\,dr\bigg)^{2}\,d\tau\notag\\
				&\lesssim\big(\phi(\varrho^{-2})\big)^{\tilde{\delta}_{0}-2}\int_{\lambda(4\varrho)}^{\infty}\tau^{-2\sigma_{1}}\,d\tau
				\lesssim 1,
			\end{align}
			Note that $\big(\phi^{-1}(r^{-1})\big)^{\frac{d}{2}}\lesssim \varrho^{-d}\big(\phi(\varrho^{-2})\big)^{-\frac{d}{2}}r^{-\frac{d}{2}}$, $\alpha(1-\tilde{\delta}_{0}/2)-\sigma_{1}=-1/2$, we derive
			\begin{align}\label{Min 3}
				&\int_{\lambda(4\varrho)}^{\infty}\bigg(\int_{B_{4\varrho}}\int_{(\phi(\varrho^{-2}/16))^{-1}}^{2\tau^{\alpha}}
				\big(\phi^{-1}(r^{-1})\big)^{\frac{d}{2}}r^{-\frac{\tilde{\delta}_{0}}{2}}\tau^{-\sigma_{1}}\,dr\,dz\bigg)^{2}\,d\tau\notag\\
				&\lesssim\int_{\lambda(4\varrho)}^{\infty}\tau^{-2\sigma_{1}}\phi(\varrho^{-2})^{-d}
				\bigg(\int_{(\phi(\varrho^{-2}/16))^{-1}}^{2\tau^{\alpha}}r^{-\frac{d}{2}-\frac{\tilde{\delta}_{0}}{2}}\,dr\bigg)^{2}\,d\tau\notag\\
				&\lesssim\int_{\lambda(4\varrho)}^{\infty}\tau^{-2\sigma_{1}}\phi(\varrho^{-2})^{-d}\big[\tau^{2\alpha-\alpha d-\alpha\tilde{\delta}_{0}}
				+\big(\phi(\varrho^{-2})\big)^{-2+d+\tilde{\delta}_{0}}+\textbf{I}_{d+\tilde{\delta}_{0}=2}
				\tau^{2\alpha\varepsilon}\big(\phi(\varrho^{-2})\big)^{2\varepsilon}\big]\,d\tau\notag\\
				&\lesssim 1.
			\end{align}
			where $\varepsilon>0$ is small enough, such that $2\sigma_{1}>1+2\alpha\varepsilon$. Thus, by using Minkowski's inequality and combining with \eqref{Min 1}, \eqref{Min 2}, \eqref{Min 3}, we obtain $\big|\mathbb{S}h_{4}(t,x)\big|\lesssim\big\|h\big\|_{L_{\infty}(\mathbb{R}^{d+1};H)}$.
			
			Thus, we obtain that
			$$
			\fint_{Q_{\varrho}}\fint_{Q_{\varrho}}\big|\mathbb{S}h_{4}(s,x)-\mathbb{S}h_{4}(s,y)\big|\,dxdtdsdy\lesssim \big\|h\big\|_{L_{\infty}(\mathbb{R}^{d+1};H)}.
			$$
			Therefore, from Step 1--Step 4, we obtain \eqref{transform inequality}, which also implies that \eqref{BMO estimate} holds.
		\end{proof}
		
		\emph{The proof of Lemma \rm\ref{strong type p}}. We consider the Fefferman-Stein function of $\mathbb{S}h$, which is defined as
		$$
		\big(\mathbb{S}h\big)^{\sharp}(t,x)=\sup_{(t,x)\in Q_{\varrho}}\fint_{Q_{\varrho}}\left|\mathbb{S}h(s,y)-\big(\mathbb{S}h\big)_{Q_{\varrho}}\right|\,dyds.
		$$
		Obviously, $\big(\mathbb{S}h\big)^{\sharp}$ is sublinear in $h$, and by the Fefferman-Stein theorem \cite{Grafakos}, for $1<p<\infty$,
		$$
		\big\|h\big\|_{L_{p}(\mathbb{R}^{d+1};H)}\lesssim\big\|\big(h\big)^{\sharp}\big\|_{L_{p}(\mathbb{R}^{d+1};H)}\lesssim
		\big\|h\big\|_{L_{p}(\mathbb{R}^{d+1};H)}.
		$$
		Since the operator $\mathbb{S}$ is strong type $(2,2)$, this implies that
		$$
		\big\|\big(\mathbb{S}h\big)^{\sharp}\big\|_{L_{2}(\mathbb{R}^{d+1})}\lesssim \big\|h\big\|_{L_{2}(\mathbb{R}^{d+1};H)},
		$$
		Proposition \ref{dengjiaxingzhi} implies that
		$$
		\big\|\big(\mathbb{S}h\big)^{\sharp}\big\|_{L_{\infty}(\mathbb{R}^{d+1})}\lesssim \big\|h\big\|_{L_{\infty}(\mathbb{R}^{d+1};H)},
		$$
		Then by Marcinkiewicz interpolation, for any $2<p<\infty$,
		$$
		\big\|\big(\mathbb{S}h\big)^{\sharp}\big\|_{L_{p}(\mathbb{R}^{d+1})}\lesssim \big\|h\big\|_{L_{p}(\mathbb{R}^{d+1};H)}.
		$$
	\end{proof}
	Now we consider L\'{e}vy process $Z^{k}_{t}$ and consider the function
	$$
	w(t,x)=\sum_{k=1}^{\infty}\int_{0}^{t}\int_{\mathbb{R}^{d}}\mathcal{S}_{\alpha,\sigma_{2}}(t-s,x-y)f^{k}(s,y)\,dydZ^{k}_{s},
	$$
	for any $c\geq 0$, by using Burkholder-Davis-Gundy inequality and Remark \ref{some remark about yang}, we derive
	\begin{align}\label{B-D-G estimate}
		&\big\|\big(\phi(-\Delta)\big)^{c}w\big\|_{\mathscr{L}_{p}(T)}\notag\\
		&\lesssim\bigg\|\bigg(\sum_{k=1}^{\infty}
		\int_{0}^{t}\int_{\mathbb{R}^{d_{1}}}\big|\big(\phi(-\Delta)\big)^{c}\mathcal{S}_{\alpha,\sigma_{2}}(t-s)\star f^{k}(s)\big|^{2}|y|^{2}\Pi_{x}(\,ds,\,dy)\bigg)^{\frac{1}{2}}\bigg\|_{L_{p}([0,T]\times\Omega;L_{p})}\notag\\
		&\leq C\sum_{r=1}^{d_{1}}\bigg\|\int_{0}^{t}\big|\big(\phi(-\Delta)\big)^{c}\mathcal{S}_{\alpha,\sigma_{2}}(t-s)\star f^{r}(s)\big|^{p}\,ds\bigg\|_{L_{1}\big([0,T]\times\Omega;L_{1}(l_{2})\big)},
	\end{align}
	the constant $C$ is dependent of $T$.
	\begin{lemma}\label{Besov control Lemma}
		For the constant $p>2$, $\varepsilon,\delta>0$ and satisfy
		$$
		\frac{1}{p}<\sigma_{2}-\frac{\alpha}{2}\varepsilon,\text{ }\sigma_{2}-\alpha+\delta<\frac{1}{p},\text{ }\delta<\frac{1}{p},
		$$
		there exist constant $C$ is dependent of $\alpha,\sigma_{2},p,d,\varepsilon,\delta,T$ such that
		\begin{align*}
			\int_{0}^{T}\int_{\mathbb{R}^{d}}\int_{0}^{t}\bigg|\big(\phi(-\Delta)\big)^{\frac{\tilde{\delta}_{1}+\varepsilon}{2}}
			\mathcal{S}_{\alpha,\sigma_{2}}(t-s)\star f(s)\bigg|^{p}\,ds\,dxdt
			\leq C\int_{0}^{T}\big\|f(t)\big\|^{p}_{B^{\phi,\varepsilon}_{p,p}}\,dt.
		\end{align*}
	\end{lemma}
	\begin{proof}
		We introduce the Littlewood-Paley decomposition,
		$$
		\Delta_{j}=\big(\Delta_{j-1}+\Delta_{j}+\Delta_{j+1}\big)\Delta_{j},\text{ }j=\pm1,\pm2,...,\text{ }\Delta_{0}=\big(\Delta_{0}+\Delta_{1}\big)\Delta_{0}.
		$$
		where
		$$
		\Delta_{j}=\mathcal{F}^{-1}\big(\psi(2^{-j}\xi)\big),\text{ }\sum_{j\in\mathbb{Z}}\psi(2^{-j}\xi)=1,\xi\neq 0,\text{ }\psi_{0}(\xi)=1- \sum_{j=1}^{\infty}\psi(2^{-j}\xi),
		$$
		$\psi(\cdot)\in\mathcal{S}(\mathbb{R}^{d})$ and supported in the strip $\{\xi:\frac{1}{2}\leq|\xi|\leq 1\}$. For any $1\leq p,q\leq \infty$, $s\in\mathbb{R}$, we can define the general Besov space and Tribel-Lizorkin space, $B^{\phi,s}_{p,q}$, $F^{\phi,s}_{p,q}$, see in \cite{Mikulevicius,Kim2}.
		
		First, we claim the following frequency localized estimate.
		\begin{proposition}\label{pinlvjubuhua estimate}
			Under the condition of Lemma {\rm\ref{Besov control Lemma}}, we derive
			$$
			\big\|\Delta_{j}\big(\phi(-\Delta)\big)^{\frac{\tilde{\delta}_{1}+\varepsilon}{2}}\mathcal{S}_{\alpha,\sigma_{2}}(t,x)\big\|_{L_{1}}
			\lesssim
			\big(t^{-\frac{1}{p}-\frac{\alpha\varepsilon}{2}}\wedge\big(\phi(2^{2j})\big)^{\frac{\delta}{\alpha}
				+\frac{\varepsilon}{2}}t^{-\frac{1}{p}+\delta}\big),j=0,1,2,....
			$$
		\end{proposition}
		\begin{proof}
			First, the estimate $
			\big\|\Delta_{j}\big(\phi(-\Delta)\big)^{\frac{\tilde{\delta}_{1}+\varepsilon}{2}}\mathcal{S}_{\alpha,\sigma_{2}}(t,x)\big\|_{L_{1}}
			\lesssim
			t^{-\frac{1}{p}-\frac{\alpha\varepsilon}{2}}
			$
			follows from Lemma \ref{Some proposition of S}. Next,
			note that
			\begin{align*}
				\Delta_{j}\big(\phi(-\Delta)\big)^{\frac{\tilde{\delta}_{1}+\varepsilon}{2}}\mathcal{S}_{\alpha,\sigma_{2}}(t,x)
				&=\mathcal{F}^{-1}\big[\psi(2^{-j}\xi)(\phi(|\xi|^{2}))^{\frac{\tilde{\delta}_{1}+\varepsilon}{2}}
				\mathcal{F}\mathcal{S}_{\alpha,\sigma_{2}}(t,\xi)\big](x)\\
				&=2^{jd}\mathcal{F}^{-1}\big[\psi(\xi)(\phi(|2^{j}\xi|^{2}))^{\frac{\tilde{\delta}_{1}+\varepsilon}{2}}
				\mathcal{F}\mathcal{S}_{\alpha,\sigma_{2}}(t,2^{j}\xi)\big](2^{j}x),
			\end{align*}
			this implies
			\begin{align*}
				\big\|\Delta_{j}\big(\phi(-\Delta)\big)^{\frac{\tilde{\delta}_{1}+\varepsilon}{2}}\mathcal{S}_{\alpha,\sigma_{2}}\big\|_{L_{1}}
				=\big\|\bar{\Delta}_{j}\big(\phi(-\Delta)\big)^{\frac{\tilde{\delta}_{1}+\varepsilon}{2}}\mathcal{S}_{\alpha,\sigma_{2}}\big\|_{L_{1}},
			\end{align*}
			where
			$
			\mathcal{F}\big(\bar{\Delta}_{j}
			\big(\phi(-\Delta)\big)^{\frac{\tilde{\delta}_{1}+\varepsilon}{2}}\mathcal{S}_{\alpha,\sigma_{2}}\big)(t,\xi)
			=\psi(\xi)(\phi(|2^{j}\xi|^{2}))^{\frac{\tilde{\delta}_{1}+\varepsilon}{2}}
			\mathcal{F}\mathcal{S}_{\alpha,\sigma_{2}}(t,2^{j}\xi).
			$
			From \cite{Gorenflo}, we derive
			\begin{align*}
				E_{\alpha,\beta}(-z) = \int_{0}^{\infty} \frac{1}{\pi \alpha} r^{\frac{1-\beta}{\alpha}} \exp\left(-r^{\frac{1}{\alpha}}\right)
				\frac{r \sin(\pi(1-\beta)) + z \sin(\pi(1-\beta + \alpha))}{r^2 + 2rz \cos(\pi \alpha) + z^2} \, dr,\text{  }\forall z>0,\beta<1+\alpha.
			\end{align*}
			Thus, we derive
			\begin{align*}
				\big|\mathcal{F}\big(\bar{\Delta}_{j}
				\big(\phi(-\Delta)\big)^{\frac{\tilde{\delta}_{1}+\varepsilon}{2}}\mathcal{S}_{\alpha,\sigma_{2}}\big)(t,\xi)\big|&=
				t^{\alpha-\sigma_{2}}\psi(\xi)(\phi(|2^{j}\xi|^{2}))^{\frac{\tilde{\delta}_{1}+\varepsilon}{2}}
				E_{\alpha,1-\sigma_{2}+\alpha}(-t^{\alpha}\phi(|2^{j}\xi|^{2}))\\
				&\lesssim J_{1}+J_{2},
			\end{align*}
			where
			\begin{align*}
				J_{1}&=\textbf{I}_{\frac{1}{2}\leq |\xi|\leq 2}t^{\alpha-\sigma_{2}}\int_{0}^{\infty}
				\frac{\exp(-r^{\frac{1}{\alpha}})r^{\frac{\sigma_{2}}{\alpha}}\big(\phi(|2^{j}\xi|^{2})\big)^{\frac{\tilde{\delta}_{1}+\varepsilon}{2}}}
				{r^{2}+2rt^{\alpha}
					\phi(|2^{j}\xi|^{2})\cos(\alpha\pi)
					+t^{2\alpha}\big(\phi(|2^{j}\xi|^{2})\big)^{2}}\,dr,\\
				J_{2}&=\textbf{I}_{\frac{1}{2}\leq |\xi|\leq 2}t^{2\alpha-\sigma_{2}}\int_{0}^{\infty}\frac{\exp(-r^{\frac{1}{\alpha}})r^{\frac{\sigma_{2}}{\alpha}-1}
					\big(\phi(|2^{j}\xi|^{2})\big)^{\frac{\tilde{\delta}_{1}+\varepsilon}{2}+1}}{r^{2}+2rt^{\alpha}
					\phi(|2^{j}\xi|^{2})\cos(\alpha\pi)
					+t^{2\alpha}\big(\phi(|2^{j}\xi|^{2})\big)^{2}}\,dr.
			\end{align*}
			By change variable $r\leftrightarrow r^{\alpha}t^{\alpha}\phi(|2^{j}\xi|^{2})$, and note $\sigma_{2}/\alpha+(\tilde{\sigma}_{1}+\varepsilon)/2-1=1/\alpha p+\varepsilon/2$, we derive
			\begin{align}\label{estimate J1}
				J_{1}&=\textbf{I}_{\frac{1}{2}\leq |\xi|\leq 2}\big(\phi(|2^{j}\xi|^{2})\big)^{\frac{1}{\alpha p}+\frac{\varepsilon}{2}}\int_{0}^{\infty}\frac{\exp\big(-rt\big(\phi(|2^{j}\xi|^{2})\big)^{\frac{1}{\alpha}}\big)
					r^{\alpha+\sigma_{2}-1}}{r^{2\alpha}+2r\cos(\alpha\pi)+1}\,dr\notag\\
				&\lesssim\textbf{I}_{\frac{1}{2}\leq |\xi|\leq 2}\big(\phi(|2^{j}\xi|^{2})\big)^{\frac{1}{\alpha p}+\frac{\varepsilon}{2}}
				\bigg[\int_{0}^{1}r^{\alpha+\sigma_{2}-\frac{1}{p}+\delta-1}
				\big[t\big(\phi(|2^{j}\xi|^{2})\big)^{\frac{1}{\alpha}}\big]^{-\frac{1}{p}+\delta}\,dr\notag\\
				&\quad+\int_{1}^{\infty}r^{\sigma_{2}-\frac{1}{p}+\delta-1-\alpha}
				\big[t\big(\phi(|2^{j}\xi|^{2})\big)^{\frac{1}{\alpha}}\big]^{-\frac{1}{p}+\delta}\,dr\bigg]\notag\\
				&\lesssim\big(\phi(2^{2j})\big)^{\frac{\delta}{\alpha}
					+\frac{\varepsilon}{2}}t^{-\frac{1}{p}+\delta},
			\end{align}
			\begin{align}\label{estimate J2}
				J_{2}&=\textbf{I}_{\frac{1}{2}\leq |\xi|\leq 2}\big(\phi(|2^{j}\xi|^{2})\big)^{\frac{1}{\alpha p}+\frac{\varepsilon}{2}}\int_{0}^{\infty}\frac{\exp\big(-rt\big(\phi(|2^{j}\xi|^{2})\big)^{\frac{1}{\alpha}}\big)
					r^{\sigma_{2}-1}}{r^{2\alpha}+2r\cos(\alpha\pi)+1}\,dr\notag\\
				&\lesssim\textbf{I}_{\frac{1}{2}\leq |\xi|\leq 2}\big(\phi(|2^{j}\xi|^{2})\big)^{\frac{1}{\alpha p}+\frac{\varepsilon}{2}}
				\bigg[\int_{0}^{1}r^{\sigma_{2}-\frac{1}{p}+\delta-1}
				\big[t\big(\phi(|2^{j}\xi|^{2})\big)^{\frac{1}{\alpha}}\big]^{-\frac{1}{p}+\delta}\,dr\notag\\
				&\quad+\int_{1}^{\infty}r^{\sigma_{2}-\frac{1}{p}+\delta-1-2\alpha}
				\big[t\big(\phi(|2^{j}\xi|^{2})\big)^{\frac{1}{\alpha}}\big]^{-\frac{1}{p}+\delta}\,dr\bigg]\notag\\
				&\lesssim\big(\phi(2^{2j})\big)^{\frac{\delta}{\alpha}
					+\frac{\varepsilon}{2}}t^{-\frac{1}{p}+\delta}
			\end{align}
			For any multi-index $\gamma$, $D^{\gamma}_{\xi}\psi$ is also Schwartz function, and
			$$
			D^{\gamma}_{\xi}\phi(|2^{j}\xi|^{2})=\sum_{\frac{|\gamma|}{2}\leq k\leq |\gamma|}\big(2^{j}\big)^{2k-|\gamma|}\phi^{(k)}(|2^{j}\xi|^{2})\prod_{i=1}^{d}|\xi_{i}|^{\beta_{i}},\quad\text{where }\sum_{i=1}^{d}\beta_{i}=2k-|\gamma|,
			$$
			combine \eqref{B.S.T,func.bdd} we obtain
			$$
			\left|D^{\gamma}_{\xi}\phi(|2^{j}\xi|^{2})\right|\lesssim 2^{-j|\gamma|}\left|\xi\right|^{-|\gamma|}\phi(|2^{j}\xi|^{2}).
			$$
			By the Leibniz rule, we obtain
			\begin{align*}
				\bigg|D^{\gamma}_{\xi}\big(\phi(|2^{j}\xi|^{2})\big)^{\frac{\tilde{\delta}_{1}+\varepsilon}{2}}\bigg|&\lesssim
				\bigg|\sum_{\substack{\gamma_{1}+\gamma_{2}+...+\gamma_{l}=\gamma,\\ 1\leq l\leq|\gamma|}}\big(\phi(|2^{j}\xi|^{2})\big)^{\frac{\tilde{\delta}_{1}+\varepsilon}{2}-l}
				\prod_{i=1}^{l}D^{\gamma_{i}}_{\xi}\phi(|2^{j}\xi|^{2})\bigg|\\
				&\lesssim 2^{-j|\gamma|}\big|\xi\big|^{-|\gamma|}\big(\phi(|2^{j}\xi|^{2})\big)^{\frac{\tilde{\delta}_{1}+\varepsilon}{2}},
			\end{align*}
			thus we obtain
			\begin{align*}
				&\big|D_{\xi}^{\gamma}J_{1}\big|\\&\lesssim\textbf{I}_{\frac{1}{2}\leq |\xi|\leq 2}t^{\alpha-\sigma_{2}}\sum_{\substack{\gamma_{1}+\gamma_{2}=\gamma,\\\beta_{1}+...+\beta_{l}=\gamma_{2},\\ 1\leq l\leq|\gamma_{2}|}}\int_{0}^{\infty}\exp(-r^{\frac{1}{\alpha}})r^{\frac{\sigma_{2}}{\alpha}}
				\big|D_{\xi}^{\gamma_{1}}\big(\phi(|2^{j}\xi|^{2})
				\big)^{\frac{\tilde{\delta}_{1}+\varepsilon}{2}}\bigg|\frac{\prod_{i=1}^{l}
					D^{\beta_{i}}_{\xi}g(r,t,2^{j}\xi)}{[g(r,t,2^{j}\xi)]^{l+1}}\bigg|\,dr,
			\end{align*}
			where
			$$
			g(r,t,2^{j}\xi)=r^{2}+2rt^{\alpha}\phi(|2^{j}\xi|^{2})\cos(\alpha\pi)+t^{2\alpha}\big(\phi(|2^{j}\xi|^{2})\big)^{2}.
			$$
			Note that $\frac{1}{2}\leq |\xi|\leq 2$, and
			\begin{align*}
				&\bigg|\prod_{i=1}^{l}D^{\beta_{i}}_{\xi}g(r,t,2^{j}\xi)\bigg|\lesssim\prod_{i=1}^{l}
				\bigg(\textbf{I}_{\beta_{i}=0}r^{2}+2rt^{\alpha}\cos(\alpha\pi)\phi(|2^{j}\xi|^{2})|2^{j}\xi|^{-|\beta_{i}|}
				+t^{2\alpha}(\phi(|2^{j}\xi|^{2}))^{2}|2^{j}\xi|^{-|\beta_{i}|}\bigg),
			\end{align*}
			combine \eqref{estimate J1}, we derive
			\begin{align*}
				\big|D_{\xi}^{\gamma}J_{1}\big|&\lesssim\textbf{I}_{\frac{1}{2}\leq |\xi|\leq 2}t^{\alpha-\sigma_{2}}\int_{0}^{\infty}
				\frac{\exp(-r^{\frac{1}{\alpha}})r^{\frac{\sigma_{2}}{\alpha}}\big(\phi(|2^{j}\xi|^{2})\big)^{\frac{\tilde{\delta}_{1}+\varepsilon}{2}}}
				{r^{2}+2rt^{\alpha}
					\phi(|2^{j}\xi|^{2})\cos(\alpha\pi)
					+t^{2\alpha}\big(\phi(|2^{j}\xi|^{2})\big)^{2}}\,dr\\
				&\lesssim\textbf{I}_{\frac{1}{2}\leq |\xi|\leq 2}\big(\phi(|2^{j}\xi|^{2})\big)^{\frac{1}{\alpha p}+\frac{\varepsilon}{2}}
				\bigg[\int_{0}^{1}r^{\alpha+\sigma_{2}-\frac{1}{p}+\delta-1}
				\big[t\big(\phi(|2^{j}\xi|^{2})\big)^{\frac{1}{\alpha}}\big]^{-\frac{1}{p}+\delta}\,dr\notag\\
				&\quad+\int_{1}^{\infty}r^{\sigma_{2}-\frac{1}{p}+\delta-1-\alpha}
				\big[t\big(\phi(|2^{j}\xi|^{2})\big)^{\frac{1}{\alpha}}\big]^{-\frac{1}{p}+\delta}\,dr\bigg]\notag\\
				&\lesssim\big(\phi(2^{2j})\big)^{\frac{\delta}{\alpha}
					+\frac{\varepsilon}{2}}t^{-\frac{1}{p}+\delta},
			\end{align*}
			and the estimate $\big|D_{\xi}^{\gamma}J_{2}\big|\lesssim \big(\phi(2^{2j})\big)^{\frac{\delta}{\alpha}
				+\frac{\varepsilon}{2}}t^{-\frac{1}{p}+\delta}$ is similar to $\big|D_{\xi}^{\gamma}J_{1}\big|$.
			
			Therefore, for any multi-index, we obtain
			$$
			\big|D^{\gamma}_{\xi}\big[\psi(\xi)(\phi(|2^{j}\xi|^{2}))^{\frac{\tilde{\delta}_{1}+\varepsilon}{2}}
			\mathcal{F}\mathcal{S}_{\alpha,\sigma_{2}}(t,2^{j}\xi)\big]\big|\lesssim\big(\phi(2^{2j})\big)^{\frac{\delta}{\alpha}
				+\frac{\varepsilon}{2}}t^{-\frac{1}{p}+\delta},
			$$
			and we derive
			\begin{align*}
				&\big\|\bar{\Delta}_{j}\big(\phi(-\Delta)\big)^{\frac{\tilde{\delta}_{1}+\varepsilon}{2}}\mathcal{S}_{\alpha,\sigma_{2}}\big\|_{L_{1}}\\
				&\lesssim\int_{\mathbb{R}^{d}}(1+|x|^{2})^{-d}(1+|x|^{2})^{d}\big|\bar{\Delta}_{j}
				\big(\phi(-\Delta)\big)^{\frac{\tilde{\delta}_{1}+\varepsilon}{2}}\mathcal{S}_{\alpha,\sigma_{2}}(t,x)\big|\,dx\\
				&\lesssim\int_{\mathbb{R}^{d}}(1+|x|^{2})^{-d}\,dx\sup_{\xi}\big|\big(I-\Delta\big)^{d}
				[\psi(\xi)(\phi(|2^{j}\xi|^{2}))^{\frac{\tilde{\delta}_{1}+\varepsilon}{2}}
				\mathcal{F}\mathcal{S}_{\alpha,\sigma_{2}}(t,2^{j}\xi)]\big|\\
				&\lesssim\big(\phi(2^{2j})\big)^{\frac{\delta}{\alpha}
					+\frac{\varepsilon}{2}}t^{-\frac{1}{p}+\delta}.
			\end{align*}
		\end{proof}
		\textit{Proof of Lemma \rm\ref{Besov control Lemma}}.
		Note that $L_{p}\approx F^{\phi,0}_{p,2}$ for any $1<p<\infty$, we derive
		\begin{align*}
			&\int_{0}^{T}\int_{0}^{t}\big\|\big(\phi(-\Delta)\big)^{\frac{\tilde{\delta}_{1}+\varepsilon}{2}}
			\mathcal{S}_{\alpha,\sigma_{2}}(t-s)\star f(s)\big\|^{p}_{L_{p}}\,ds\,dt\\
			&\sim\int_{0}^{T}\int_{0}^{t}\big\|\Delta_{0}\big(\phi(-\Delta)\big)^{\frac{\tilde{\delta}_{1}+\varepsilon}{2}}
			\mathcal{S}_{\alpha,\sigma_{2}}(t-s)\star f(s)\big\|^{p}_{L_{p}}\,ds\,dt\\
			&\quad+\int_{0}^{T}\int_{0}^{t}\big\|\big(\sum_{j=1}^{\infty}\big|\Delta_{j}\big(\phi(-\Delta)\big)^{\frac{\tilde{\delta}_{1}+\varepsilon}{2}}
			\mathcal{S}_{\alpha,\sigma_{2}}(t-s)\star f(s)\big|^{2}\big)^{\frac{1}{2}}\big\|^{p}_{L_{p}}\,ds\,dt\triangleq I_{1}+I_{2}.
		\end{align*}
		We estimate $I_{1}$ and $I_{2}$ separately.
		
		\underline{Estimate of $I_{1}$}:\text{ } Combining with Proposition \ref{pinlvjubuhua estimate}, we derive
		\begin{align*}
			&\int_{0}^{T}\int_{0}^{t}\big\|\Delta_{0}\big(\phi(-\Delta)\big)^{\frac{\tilde{\delta}_{1}+\varepsilon}{2}}
			\mathcal{S}_{\alpha,\sigma_{2}}(t-s)\star f(s)\big\|^{p}_{L_{p}}\,ds\,dt\\
			&\lesssim\int_{0}^{T}\int_{0}^{t}\big\|\sum_{i=0}^{1}\Delta_{i}\big(\phi(-\Delta)\big)^{\frac{\tilde{\delta}_{1}+\varepsilon}{2}}
			\mathcal{S}_{\alpha,\sigma_{2}}(t-s)\star \Delta_{0}f(s)\big\|^{p}_{L_{p}}\,ds\,dt\\
			&\lesssim\int_{0}^{T}\int_{s}^{T}\big((t-s)^{-\frac{1}{p}-\frac{\alpha\varepsilon}{2}}\wedge(\phi(4))^{\frac{\delta}{\alpha}
				+\frac{\varepsilon}{2}}(t-s)^{-\frac{1}{p}+\delta}\big)^{p}\big\|\Delta_{0}f(s)\big\|_{L_{p}}^{p}\,dtds\\
			&\lesssim\int_{0}^{T}\int_{(s+(\phi(4))^{-\frac{1}{\alpha}})\wedge T}^{T}(t-s)^{-1-\frac{p\alpha\varepsilon}{2}}\big\|\Delta_{0}f(s)\big\|_{L_{p}}^{p}\,dt\,ds\\
			&\quad+\int_{0}^{T}\int_{s}^{(s+(\phi(4))^{-\frac{1}{\alpha}})\wedge T}\big(\phi(4)\big)^{\frac{p\delta}{\alpha}+\frac{p\varepsilon}{2}}(t-s)^{-1+p\delta}\big\|\Delta_{0}f(s)\big\|_{L_{p}}^{p}\,dt\,ds\\
			&\lesssim\int_{0}^{T}\big\|\Delta_{0}f(s)\big\|_{L_{p}}^{p}\,ds.
		\end{align*}
		
		\underline{Estimate of $I_{2}$}:\text{ } Combining Minkowski's inequality, Proposition \ref{pinlvjubuhua estimate}, and Fubini's theorem, we derive
		\begin{align*}
			&\int_{0}^{T}\int_{0}^{t}\big\|\big(\sum_{j=1}^{\infty}\big|\Delta_{j}\big(\phi(-\Delta)\big)^{\frac{\tilde{\delta}_{1}+\varepsilon}{2}}
			\mathcal{S}_{\alpha,\sigma_{2}}(t-s)\star f(s)\big|^{2}\big)^{\frac{1}{2}}\big\|^{p}_{L_{p}}\,ds\,dt\\
			&\lesssim\int_{0}^{T}\int_{s}^{T}\big\|\big(\sum_{j=1}^{\infty}\big|\sum_{i=j-1}^{j+1}\Delta_{i}
			\big(\phi(-\Delta)\big)^{\frac{\tilde{\delta}_{1}+\varepsilon}{2}}
			\mathcal{S}_{\alpha,\sigma_{2}}(t-s)\star \Delta_{j}f(s)\big|^{2}\big)^{\frac{1}{2}}\big\|^{p}_{L_{p}}\,ds\,dt\\
			&\lesssim\int_{0}^{T}\int_{s}^{T}\big(\sum_{j=1}^{\infty}\big\|\sum_{i=j-1}^{j+1}\Delta_{i}
			\big(\phi(-\Delta)\big)^{\frac{\tilde{\delta}_{1}+\varepsilon}{2}}
			\mathcal{S}_{\alpha,\sigma_{2}}(t-s)\star \Delta_{j}f(s)\big\|_{L_{p}}^{2}\big)^{\frac{p}{2}}\,dt\,ds\\
			&\lesssim\int_{0}^{T}\int_{s}^{T}\big(\sum_{j=1}^{\infty}\big[(t-s)^{-\frac{1}{p}-\frac{\alpha\varepsilon}{2}}
			\wedge\big(\phi(2^{2j})\big)^{\frac{\delta}{\alpha}+\frac{\varepsilon}{2}}(t-s)^{-\frac{1}{p}+\delta}
			\big]^{2}\|\Delta_{j}f(s)\|^{2}_{L_{p}}\big)^{\frac{p}{2}}\,dt\,ds\\
			&\lesssim\int_{0}^{T}\int_{(s+(\phi(2^{2j}))^{-\frac{1}{\alpha}})\wedge T}^{T}\big(\sum_{j\in J(t,s,j)}^{\infty}\|\Delta_{j}f(s)\|^{2}_{L_{p}}\big)^{\frac{p}{2}}(t-s)^{-1-\frac{p\alpha\varepsilon}{2}}\,dt\,ds\\
			&\quad+\int_{0}^{T}\int_{s}^{(s+(\phi(2^{2j}))^{-\frac{1}{\alpha}})\wedge T}\big(\sum_{j\notin J(t,s,j)}\big(\phi(2^{2j})\big)^{\frac{2\delta}{\alpha}+\varepsilon}\|\Delta_{j}f(s)\|_{L_{p}}^{2}\big)^{\frac{p}{2}}
			(t-s)^{-1+p\delta}\,dt\,ds,
		\end{align*}
		where
		\begin{align*}
			J(t,s,j)&=\big\{(t,s,j):(t-s)^{-\frac{\alpha\varepsilon}{2}-\delta}
			<\big(\phi(2^{2j})\big)^{\frac{\delta}{\alpha}+\frac{\varepsilon}{2}}\big\},\\
			(t-s)^{-\frac{\alpha\varepsilon}{2}-\delta}
			&<\big(\phi(2^{2j})\big)^{\frac{\delta}{\alpha}+\frac{\varepsilon}{2}}\Rightarrow t>s+\big(\phi(2^{2j})\big)^{-\frac{1}{\alpha}}.
		\end{align*}
		We take $a\in(0,\alpha\varepsilon)$ and use H\"{o}lder's inequality,
		\begin{align*}
			\big(\sum_{j\in J(t,s,j)}^{\infty}\|\Delta_{j}f(s)\|^{2}_{L_{p}}\big)^{\frac{p}{2}}&=\big(\sum_{j\in J(t,s,j)}(\phi(2^{2j}))^{-\frac{a}{\alpha}}(\phi(2^{2j}))^{\frac{a}{\alpha}}\|\Delta_{j}f(s)\|^{2}_{L_{p}}\big)^{\frac{p}{2}}\\
			&\leq\big(\sum_{j\geq j_{0}(t,s)}\big(\phi(2^{2j})\big)^{-\frac{ap}{\alpha(p-2)}}\big)^{\frac{p-2}{2}}\big(\sum_{j\geq j_{0}(t,s)}\big(\phi(2^{2j})\big)^{\frac{ap}{\alpha2}}\|\Delta_{j}f(s)\|^{p}_{L_{p}}\big),
		\end{align*}
		where $j_{0}(t,s)$ is the minimal integer depending on $t,s$ such that
		$
		(t-s)^{-\frac{\alpha\varepsilon}{2}-\delta}\big(\phi(2^{2j})\big)^{-\frac{2\delta+\alpha\varepsilon}{2\alpha}}<1.
		$
		Combining with \eqref{lower scailing condition}, we obtain
		$$
		\big(\sum_{j\geq j_{0}(t,s)}\big(\phi(2^{2j})\big)^{-\frac{ap}{\alpha(p-2)}}\big)^{\frac{p-2}{2}}\leq (t-s)^{\frac{ap}{2}},
		$$
		\begin{align*}
			&\int_{0}^{T}\int_{(s+(\phi(2^{2j}))^{-\frac{1}{\alpha}})\wedge T}^{T}\big(\sum_{j\in J(t,s,j)}^{\infty}\|\Delta_{j}f(s)\|^{2}_{L_{p}}\big)^{\frac{p}{2}}(t-s)^{-1-\frac{p\alpha\varepsilon}{2}}\,dt\,ds\\
			&\lesssim\int_{0}^{T}\int_{(s+(\phi(2^{2j}))^{-\frac{1}{\alpha}})\wedge T}^{T}\big(\sum_{j\geq j_{0}(t,s)}\big(\phi(2^{2j})\big)^{\frac{ap}{\alpha2}}\|\Delta_{j}f(s)\|^{p}_{L_{p}}\big)
			(t-s)^{-1-\frac{p\alpha\varepsilon}{2}+\frac{ap}{2}}\,dtds\\
			&\lesssim\int_{0}^{T}\sum_{j\geq j_{0}(t,s)}\big(\phi(2^{2j})\big)^{\frac{p\varepsilon}{2}}\|\Delta_{j}f(s)\|_{L_{p}}^{p}\,ds.
		\end{align*}
		We take $0<b<2\delta$, and use H\"{o}lder's inequality, we get
		\begin{align*}
			&\int_{0}^{T}\int_{s}^{(s+(\phi(2^{2j}))^{-\frac{1}{\alpha}})\wedge T}\big(\sum_{j\notin J(t,s,j)}\big(\phi(2^{2j})\big)^{\frac{2\delta}{\alpha}+\varepsilon}\|\Delta_{j}f(s)\|_{L_{p}}^{2}\big)^{\frac{p}{2}}
			(t-s)^{-1+p\delta}\,dt\,ds\\
			&\lesssim\int_{0}^{T}\int_{s}^{(s+(\phi(2^{2j}))^{-\frac{1}{\alpha}})\wedge T}
			\big(\sum_{j\leq j_{0}(t,s)}(\phi(2^{2j}))^{\frac{b}{\alpha}}
			\big(\phi(2^{2j})\big)^{\frac{2\delta}{\alpha}+\varepsilon-\frac{b}{\alpha}}\|\Delta_{j}f(s)\|_{L_{p}}^{2}\big)^{\frac{p}{2}}
			(t-s)^{-1+p\delta}\,dt\,ds\\
			&\lesssim\int_{0}^{T}\int_{s}^{(s+(\phi(2^{2j}))^{-\frac{1}{\alpha}})\wedge T}\big(\sum_{j\leq j_{0}(t,s)}(\phi(2^{2j}))^{\frac{bp}{\alpha(p-2)}}\big)^{\frac{p-2}{2}}\\
			&\quad\times\big(\sum_{j\leq j_{0}(t,s)}\big(\phi(2^{2j})^{\frac{(2\delta+\alpha\varepsilon)p}{2\alpha}-\frac{bp}{2\alpha}}\big)\|\Delta_{j}f(s)\|_{L_{p}}^{p}\big)
			(t-s)^{-1+p\delta}\,dtds\\
			&\lesssim\int_{0}^{T}\int_{s}^{(s+(\phi(2^{2j}))^{-\frac{1}{\alpha}})\wedge T}\big(\sum_{j\leq j_{0}(t,s)}\big(\phi(2^{2j})^{\frac{(2\delta+\alpha\varepsilon)p}{2\alpha}-\frac{bp}{2\alpha}}\big)\|\Delta_{j}f(s)\|_{L_{p}}^{p}\big)
			(t-s)^{-\frac{ap}{2}-1+p\delta}\,dtds\\
			&\lesssim\int_{0}^{T}\sum_{j\leq j_{0}(t,s)}\big(\phi(2^{2j})\big)^{\frac{p\varepsilon}{2}}\|\Delta_{j}f(s)\|_{L_{p}}^{p}\,ds.
		\end{align*}
		Combining the estimates for $I_{1}$ and $I_{2}$, we derive Lemma \ref{Besov control Lemma}.
	\end{proof}
	\begin{lemma}{\rm\cite[Theorem 5.3]{Kim2}}
		For $p>1$, $f\in C_{c}^{\infty}(\mathbb{R}^{d})$
		$$
		\int_{0}^{T}\big\|\mathcal{S}_{\alpha,\alpha}(t)\star f\big\|_{L_{p}}^{p}\,dt\leq C\big\|f\big\|_{B_{p,p}^{\phi,-\frac{2}{\alpha p}}},
		$$
		where the constant $C$ is dependent of $\alpha,\sigma_{2},p,d,\delta,T$.
	\end{lemma}
	
	\subsection{The regularity result}
	Next, we establish the regularity result for the \eqref{TSFSPDE}.
	
	We first prove an auxiliary Lemma.
	\begin{lemma}\label{fuzhuyinli}
		Let $p\geq 2$, $\gamma\in\mathbb{R}$, $\sigma_{1}<\alpha+\frac{1}{2}$, $\sigma_{2}<\alpha+\frac{1}{p}$, then for $g\in \mathscr{H}_{p}^{\phi,\gamma}(T)$, $h\in\mathscr{H}^{\phi,\gamma+\delta_{0}}_{p}(T,l_{2})$, and $f\in\mathscr{H}^{\phi,\gamma+\delta_{1}}_{p}(T,l_{2},d_{1})$, and $w_{0}\in \mathbb{B}_{p,p}^{\phi,\gamma+2-\frac{2}{\alpha p}}$, the linear equation
		\begin{align}\label{linear equation SPDE}
			\partial_t^\alpha w = \phi(\Delta)w +g(t,x)+\sum_{k=1}^{\infty}\partial_t^{\sigma_{1}}\int_{0}^{t}h^{k}(t,x)\,dB^{k}_{s}
			+\sum_{k=1}^{\infty}\partial_t^{\sigma_{2}}\displaystyle\int_{0}^{t}f^{k}(s,x)\,dZ^{k}_{s},\text{ }
			w(0) =\textbf{I}_{\alpha p>1}w_{0}
		\end{align}
		has the unique solution $w\in \mathcal{H}^{\phi,\gamma+2}_{p}(T)$ and satisfy
		\begin{align}\label{prior estimate}
			\big\|w\big\|_{\mathcal{H}^{\phi,\gamma+2}_{p}(T)}\leq C\big(\textbf{I}_{\alpha p>1}\|w_{0}\|_{\mathbb{B}_{p,p}^{\phi,\gamma+2-\frac{2}{\alpha p}}}+
			\|g\|_{\mathscr{H}_{p}^{\phi,\gamma}(T)}+\|h\|_{\mathscr{H}_{p}^{\phi,\gamma+\delta_{0}}(T,l_{2})}
			+\|f\|_{\mathscr{H}_{p}^{\phi,\gamma+\delta_{1}}(T,l_{2},d_{1})}\big),
		\end{align}
		where the constant $C$ is dependent of $\alpha,\gamma,p,d,\sigma_{1},\sigma_{2}, T$.
	\end{lemma}
	\begin{proof}
		Note that $\left(I-\phi(\Delta)\right)^{\frac{\nu}{2}}$ is an isometric isomorphism mapping from $\mathcal{H}^{\phi,\gamma+2}_{p}(T)$ to $\mathcal{H}^{\phi,\gamma+2-\nu}_{p}(T)$, we only need to verify the case $\gamma=0$. We will verify the a priori estimate \eqref{prior estimate}.
		
		The case $h=f=0$ follows from \cite[Theorem 2.8]{Kim2}, and the case $g=h=0$, $w_{0}=0$ follows from \cite[Lemma 4.2]{Kim3}. Since the equation \eqref{linear equation SPDE} is linear, it suffices to verify \eqref{linear equation SPDE} for $g=h=0$, $w_{0}=0$. Note that $\mathscr{H}^{\infty}_{0}(T,l_{2},d_{1})$ is dense in $\mathscr{H}_{p}^{\phi,\gamma+\delta_{1}}(T,l_{2},d_{1})$, we only need to verify for $f\in \mathscr{H}^{\infty}_{0}(T,l_{2},d_{1})$.
		
		$\bullet$\text{ }Case $\sigma_{2}>\frac{1}{p}$: By Remark \ref{some usuful remark} (v), we have
		\begin{align*}
			\left\|w\right\|_{\mathscr{L}_{p}(T)}^{p}&\leq C\int_{0}^{T}(T-s)^{\theta-1}\big(\left\|\phi(\Delta)w\right\|^{p}_{\mathscr{L}_{p}(s)}+\left\|f\right\|
			_{\mathscr{L}_{p}(s,l_{2},d_{1})}^{p}\big)\,ds\\
			&\leq C\big(\left\|\phi(\Delta)w\right\|^{p}_{\mathscr{L}_{p}(T)}+\left\|f\right\|
			_{\mathscr{L}_{p}(T,l_{2},d_{1})}^{p}\big)
		\end{align*}
		Combining Lemma \ref{Besov control Lemma}, \eqref{B-D-G estimate}, and denoting $v=(\phi(-\Delta))^{1-\frac{\tilde{\delta}_{1}+\varepsilon}{2}}w$, $\bar{f}=(\phi(-\Delta))^{1-\frac{\tilde{\delta}_{1}+\varepsilon}{2}}f$, where $\varepsilon$ is chosen as in Lemma \ref{Besov control Lemma}, and noting that $F^{\phi,s}_{p,2}\hookrightarrow B^{\phi,s}_{p,p}$ for any $s\in\mathbb{R}$, we derive
		\begin{align*}
			\|\phi(\Delta)w\|_{\mathscr{L}_{p}(T)}&=\|(\phi(-\Delta))^{\frac{\tilde{\delta}_{1}+\varepsilon}{2}}v\|_{\mathscr{L}_{p}(T)}\\
			&\leq C\sum_{r=1}^{d_{1}}\bigg\|\int_{0}^{t}\big|(\phi(-\Delta))^{\frac{\tilde{\delta}_{1}+\varepsilon}{2}}
			\mathcal{S}_{\alpha,\sigma_{2}}(t-s)\star \bar{f}^{r}(s)\big|^{p}\,ds\bigg\|_{L_{1}\big([0,T]\times\Omega;L_{1}(l_{2})\big)}\\
			&\leq C\sum_{r=1}^{d_{1}}\mathbb{E}\int_{0}^{T}\|\bar{f}^{r}(s)\|^{p}_{B^{\phi,\varepsilon}_{p,p}(l_{2})}\,ds
			\leq C\|f\|_{\mathscr{H}^{\phi,\delta_{1}}_{p}(T,l_{2},d_{1})},
		\end{align*}
		thus we obtain
		$$
		\|w\|^{p}_{\mathscr{L}_{p}(T)}\leq C\big(\|\phi(\Delta)w\|^{p}_{\mathscr{L}_{p}(T)}+\left\|f\right\|
		_{\mathscr{L}_{p}(T,l_{2},d_{1})}^{p}\big)\leq C\|f\|^{p}_{\mathscr{H}^{\phi,\delta_{1}}_{p}(T,l_{2},d_{1})}.
		$$
		
		$\bullet$\text{ }Case $\sigma_{2}<\frac{1}{p}$:\text{ } Here, $\delta_{1}=0$, and Equation \eqref{linear equation SPDE} becomes
		$$
		\partial^{\alpha}_{t}w=\phi(\Delta)w+\bar{f}(t,x),\quad \bar{f}(t,x)=\sum_{k=1}^{\infty}\int_{0}^{t}(t-s)^{-\sigma_{2}}f^{k}(s,x)\,dZ^{k}_{s}
		$$
		Using \cite[Theorem 2.8]{Kim2} and the Burkholder-Davis-Gundy inequality,
		\begin{align*}
			\mathbb{E}\big\|\phi(\Delta)w\big\|^{p}_{L_{p}([0,T];L_{p})}\leq C\mathbb{E}\big\|\bar{f}\big\|^{p}_{L_{p}([0,T];L_{p})}
			\leq C\big\|f\big\|^{p}_{\mathscr{L}_{p}(T,l_{2},d_{1})}.
		\end{align*}
		
		$\bullet$\text{ }Case $\sigma_{2}=\frac{1}{p}$:\text{ } Note that $\delta_{1}=\kappa$. Set $\sigma_{2}'=\frac{1}{p}+\frac{\alpha \kappa}{2}>\sigma_{2}$, and consider
		$$
		\partial^{\alpha}_{t}v=\phi(\Delta)v+\sum_{k=1}^{\infty}\partial^{\sigma_{2}'}_{t}\int_{0}^{t}f^{k}(s,x)\,dZ_{s}^{k},\text{ }v(0)=0,
		$$
		Note that $\delta_{1}'=(2\sigma_{2}'-2/p)/\alpha=\delta_{1}=\kappa$. Repeating the argument for the case $\sigma_{2}>1/p$, we derive
		$$
		\|v\|^{p}_{\mathscr{L}_{p}(T)}\leq C\big(\|\phi(\Delta)v\|^{p}_{\mathscr{L}_{p}(T)}+\left\|f\right\|
		_{\mathscr{L}_{p}(T,l_{2},d_{1})}^{p}\big)\leq C\|f\|^{p}_{\mathscr{H}^{\phi,\delta_{1}}_{p}(T,l_{2},d_{1})}.
		$$
		Note that $I^{\kappa}_{t}v$ also satisfies equation \eqref{linear equation SPDE} and the solution is unique,
		hence we obtain
		$$
		\|w\|^{p}_{\mathscr{L}_{p}(T)}=\|I^{\kappa}_{t}v\|^{p}_{\mathscr{L}_{p}(T)}\leq C\|v\|^{p}_{\mathscr{L}_{p}(T)}\leq C\|f\|^{p}_{\mathscr{H}^{\phi,\delta_{1}}_{p}(T,l_{2},d_{1})}.
		$$
	\end{proof}
	We now establish the regularity result for the \eqref{TSFSPDE}. For the nonlinear functions $g$, $h$, and $f$, we adopt assumptions analogous to those employed by K.H. Kim \cite{Kim3, Kim6}.
\begin{assumption}\label{assumption of nonlinear function}
		For any $t\in [0,T]$, $\omega\in \Omega$, $w,v\in \mathscr{H}^{\phi,\gamma+2}_{p}(T)$, we assume that $g(w)\in\mathscr{H}^{\phi,\gamma}_{p}(T,l_{2})$, $h(w)\in\mathscr{H}^{\phi,\gamma+\delta_{0}}_{p}(T,l_{2})$, $f(w)\in\mathscr{H}^{\phi,\gamma+\delta_{0}}_{p}(T,l_{2},d_{1})$. Moreover, for any $\varepsilon>0$, there exists a constant $N(\varepsilon)$ such that
		\begin{align*}
			&\|g(t,w)-g(t,v)\|_{H^{\phi,\gamma}_{p}}+\|h(t,w)-h(t,v)\|_{H^{\phi,\gamma+\delta_{0}}_{p}(l_{2})}
			+\|f(t,w)-f(t,v)\|_{H^{\phi,\gamma+\delta_{1}}_{p}(l_{2},d_{1})}\\
			&\leq\varepsilon\|w-v\|_{H^{\phi,\gamma+2}_{p}}+N\|w-v\|_{H^{\phi,\gamma}_{p}}.
		\end{align*}
	\end{assumption}
	
	\begin{theorem}\label{main theorem 1}
		For $T\in(0,\infty)$, $p\geq 2$, $\alpha\in (0,1)$, $\sigma_{1}<\alpha+\frac{1}{2}$, $\sigma_{2}<\alpha+\frac{1}{p}$, $\gamma\in\mathbb{R}$, and under Assumption \ref{assumption of nonlinear function}, the \eqref{TSFSPDE} admits a unique solution $w\in \mathcal{H}^{\phi,\gamma+2}_{p}(T)$ satisfying the estimate
		\begin{align*}
			\|w\|_{\mathcal{H}^{\phi,\gamma+2}_{p}(T)}\leq C\big(\textbf{I}_{\alpha p>1}\|w_{0}\|_{\mathbb{B}^{\phi,\gamma+2-\frac{2}{\alpha p}}_{p,p}}+\|g(0)\|_{\mathscr{H}^{\phi,\gamma}_{p}(T)}+\|h(0)\|_{\mathscr{H}_{p}^{\phi,\gamma+\delta_{0}}(T,l_{2})}
			+\|f(0)\|_{\mathscr{H}_{p}^{\phi,\gamma+\delta_{1}}(T,l_{2},d_{1})}\big),
		\end{align*}
		where the constant $C$ depends on $\alpha,\gamma,p,\sigma_{1},\sigma_{2},\delta,T$.
	\end{theorem}
	
	\begin{proof}
		\emph{Case of linear functions.} When $g,h,f$ are independent of $w$, i.e., $g(t,x,w)=g(t,x):=g(0)$, $h(t,x,w)=h(t,x):=h(0)$, and $f(t,x,w)=f(t,x):=f(0)$, Theorem \ref{main theorem 1} follows directly from Lemma \ref{linear equation}.
		
		\emph{Case of nonlinear functions.}
		Consider the equation
		\[
		\partial_t^\alpha w = \phi(\Delta)w +g(w)+\sum_{k=1}^{\infty}\partial_t^{\sigma_{1}}\int_{0}^{t}h^{k}(w)\,dB^{k}_{s}
		+\sum_{k=1}^{\infty}\partial_t^{\sigma_{2}}\int_{0}^{t}f^{k}(w)\,dZ^{k}_{s},\quad t>0;
		\]
		with initial condition $w(0) =\textbf{I}_{\alpha p>1}w_{0}$.
		
		Let $w_{1},w_{2}\in\mathcal{H}^{\phi,\gamma+2}_{p}(T)$ be two solutions to the above equation. Then $\tilde{w}=w_{1}-w_{2}$ satisfies
		\begin{align*}
			\partial_t^\alpha \tilde{w} &= \phi(\Delta)\tilde{w} +g(w_{1})-g(w_{2})+\sum_{k=1}^{\infty}\partial_t^{\sigma_{1}}\int_{0}^{t}[h^{k}(w_{1})-h^{k}(w_{2})]\,dB^{k}_{s}\\
			&\quad+\sum_{k=1}^{\infty}\partial_t^{\sigma_{2}}\int_{0}^{t}[f^{k}(w_{1})-f^{k}(w_{2})]\,dZ^{k}_{s},\quad t>0;
		\end{align*}
		with $\tilde{w}(0) =0$.
		
		Applying Lemma \ref{linear equation SPDE}, Assumption \ref{assumption of nonlinear function}, and Remark \ref{some usuful remark}, we derive that for any $t\leq T$,
		\begin{align*}
			&\|\tilde{w}\|^{p}_{\mathscr{H}_{p}^{\phi,\gamma+2}(t)}\\
			&\leq C\big(\|g(w_{1})-g(w_{2})\|^{p}_{\mathscr{H}_{p}^{\phi,\gamma}(t)}
			+\|h(w_{1})-h(w_{2})\|^{p}_{\mathscr{H}_{p}^{\phi,\gamma+\delta_{0}}(t,l_{2})}\\
			&\quad+\|f(w_{1})-f(w_{2})\|^{p}_{\mathscr{H}_{p}^{\phi,\gamma+\delta_{1}}(t,l_{2},d_{1})}\big)\\
			&\lesssim_{T}\varepsilon^{p}\|w_{1}-w_{2}\|^{p}_{\mathscr{H}^{\phi,\gamma+2}_{p}(t)}
			+N(\varepsilon)\|w_{1}-w_{2}\|^{p}_{\mathscr{H}^{\phi,\gamma}_{p}(t)}\\
			&\lesssim_{T}\varepsilon^{p}\|w_{1}-w_{2}\|^{p}_{\mathscr{H}^{\phi,\gamma+2}_{p}(t)}
			+N(\varepsilon)\int_{0}^{t}(t-s)^{\theta-1}\big(\|\phi(\Delta)\tilde{w}\|^{p}_{\mathscr{H}_{p}^{\phi,\gamma}(s)}\\
			&\quad+\|g(w_{1})-g(w_{2})\|^{p}_{\mathscr{H}_{p}^{\phi,\gamma}(s)}
			+\|h(w_{1})-h(w_{2})\|^{p}_{\mathscr{H}_{p}^{\phi,\gamma+\delta_{0}}(s,l_{2})}\\
			&\quad+\|f(w_{1})-f(w_{2})\|^{p}_{\mathscr{H}_{p}^{\phi,\gamma+\delta_{1}}(s,l_{2},d_{1})}\big)\,ds\\
			&\lesssim_{T}\varepsilon^{p}\|w_{1}-w_{2}\|^{p}_{\mathscr{H}^{\phi,\gamma+2}_{p}(t)}
			+N(\varepsilon)\int_{0}^{t}(t-s)^{\theta-1}\|w_{1}-w_{2}\|_{\mathscr{H}_{p}^{\phi,\gamma+2}(s)}\,ds.
		\end{align*}
		Then by the generalized Gronwall's inequality, for any $t\leq T$, $\|w_{1}-w_{2}\|_{\mathscr{H}_{p}^{\phi,\gamma+2}(t)}=0$, which establishes uniqueness.
		
		Next, we prove existence and the a priori estimate. Let $w^{0}\in \mathcal{H}^{\phi,\gamma+2}_{p}(T)$ be the unique solution to \eqref{TSFSPDE} with linear functions. For any $i\geq 0$, define $w^{i+1}\in\mathcal{H}^{\phi,\gamma+2}_{p}(T)$ by
		\begin{align}\label{solution sequence}
			\partial^{\alpha}_{t}w^{i+1}&=\phi(\Delta)w^{i+1}+g(w^{i})+\sum_{k=1}^{\infty}\partial_t^{\sigma_{1}}\int_{0}^{t}h^{k}(w^{i})\,dB^{k}_{s}
			\notag\\
			&\quad+\sum_{k=1}^{\infty}\partial_t^{\sigma_{2}}\int_{0}^{t}f^{k}(w^{i})\,dZ^{k}_{s},\quad
			w^{i}(0) =\textbf{I}_{\alpha p>1}w_{0}.
		\end{align}
		Then $\tilde{w}^{i}=w^{i+1}-w^{i}\in\mathcal{H}^{\phi,\gamma+2}_{p}(T)$ satisfies
		\begin{align*}
			\partial^{\alpha}_{t}\tilde{w}^{i}&=\phi(\Delta)\tilde{w}^{i}+g(w^{i})-g(w^{i-1})
			+\sum_{k=1}^{\infty}\partial_t^{\sigma_{1}}\int_{0}^{t}[h^{k}(w^{i})-h^{k}(w^{i-1})]\,dB^{k}_{s}\\
			&\quad+\sum_{k=1}^{\infty}\partial_t^{\sigma_{2}}\int_{0}^{t}[f^{k}(w^{i})-f^{k}(w^{i-1})]\,dZ^{k}_{s},\quad
			\tilde{w}^{i}(0,\cdot)=0.
		\end{align*}
		Applying Theorem \ref{fuzhuyinli}, Assumption \ref{assumption of nonlinear function} and Remark \ref{some usuful remark}, for any $t\leq T$, we derive
		\begin{align*}
			&\|w^{i+1}-w^{i}\|^{p}_{\mathcal{H}^{\phi,\gamma+2}_{p}(t)}\\
			&\leq C\big(\|g(w^{i})-g(w^{i-1})\|^{p}_{\mathscr{H}^{\phi,\gamma}_{p}(t)}
			+\|h(w^{i})-h(w^{i-1})\|^{p}_{\mathscr{H}^{\phi,\gamma+\delta_{0}}_{p}(t)}\\
			&\quad+\|f(w^{i})-f(w^{i-1})\|^{p}_{\mathscr{H}^{\phi,\gamma+\delta_{1}}_{p}(t)}\big)\\
			&\leq C\varepsilon^{p}\|w^{i}-w^{i-1}\|^{p}_{\mathscr{H}^{\phi,\gamma+2}_{p}(t)}
			+CN(\varepsilon)\|w^{i}-w^{i-1}\|^{p}_{\mathscr{H}^{\phi,\gamma}_{p}(t)}.
		\end{align*}
		Taking $\varepsilon=1$, this implies $\|w^{i+1}-w^{i}\|^{p}_{\mathcal{H}^{\phi,\gamma+2}_{p}(t)}\leq C\|w^{i}-w^{i-1}\|^{p}_{\mathcal{H}^{\phi,\gamma+2}_{p}(t)}$. On the other hand, note that
		\begin{align*}
			&\|w^{i}-w^{i-1}\|^{p}_{\mathscr{H}^{\phi,\gamma}_{p}(t)}\\
			&\leq C\int_{0}^{t}(t-s)^{\theta-1}\big(\|\phi(\Delta)(w^{i}-w^{i-1})\|^{p}_{\mathscr{H}^{\phi,\gamma}(s)}
			+\|g(w^{i-1})-g(w^{i-2})\|^{p}_{\mathscr{H}^{\phi,\gamma}_{p}(s)}\\
			&\quad+\|h(w^{i-1})-h(w^{i-2})\|^{p}_{\mathscr{H}^{\phi,\gamma+\delta_{0}}_{p}(s,l_{2})}\\
			&\quad+\|f(w^{i-1})-f(w^{i-2})\|^{p}_{\mathscr{H}^{\phi,\gamma+\delta_{0}}_{p}(s,l_{2},d_{1})}\,ds\big)\\
			&\leq C\int_{0}^{t}(t-s)^{\theta-1}\|w^{i-1}-w^{i-2}\|^{p}_{\mathscr{H}^{\phi,\gamma+2}_{p}(s)}\,ds,
		\end{align*}
		where the constant $C$ depends on $\varepsilon$. For any $t\leq T$,
		\begin{align*}
			&\|w^{2i+1}-w^{2i}\|^{p}_{\mathcal{H}^{\phi,\gamma}_{p}(t)}\\
			&\leq C\varepsilon^{p}\|w^{2i-1}-w^{2i-2}\|^{p}_{\mathcal{H}^{\phi,\gamma+2}_{p}(t)}
			+CN(\varepsilon)\int_{0}^{t}(t-s)^{\theta-1}\|w^{2i-1}-w^{2i-2}\|^{p}_{\mathcal{H}^{\phi,\gamma+2}_{p}(s)}\,ds\\
			&\leq C\varepsilon^{p}\bigg(\varepsilon^{p}\|w^{2i-3}-w^{2i-4}\|^{p}_{\mathcal{H}^{\phi,\gamma+2}_{p}(t)}
			+CN(\varepsilon)\int_{0}^{t}(t-s)^{\theta-1}\|w^{2i-3}-w^{2i-4}\|^{p}_{\mathcal{H}^{\phi,\gamma+2}_{p}(s)}\,ds\bigg)\\
			&\quad+CN(\varepsilon)\varepsilon^{p}\int_{0}^{t}(t-s)^{\theta-1}\|w^{2i-3}-w^{2i-4}\|^{p}_{\mathcal{H}^{\phi,\gamma+2}_{p}(s)}\,ds\\
			&\quad+(CN(\varepsilon))^{2}\int_{0}^{t}\int_{0}^{s}(t-s)^{\theta-1}(s-r)^{\theta-1}
			\|w^{2i-3}-w^{2i-4}\|^{p}_{\mathcal{H}^{\phi,\gamma+2}_{2}(r)}\,dr\,ds\\
			&\leq\cdots\\
			&\leq \sum_{k=0}^{i}C^{k}_{i}(\varepsilon^{p})^{i-k}(CN(\varepsilon)t^{\theta})^{k}\frac{(\Gamma(\theta))^{k}}{\Gamma(k\theta+1)}
			\|w^{1}-w^{0}\|^{p}_{\mathcal{H}^{\phi,\gamma+2}_{2}(t)}\\
			&\leq 2^{p}\varepsilon^{pi}\max_{k}\left(\frac{(\varepsilon^{-p}CN(\varepsilon)T^{\theta}\Gamma(\theta))^{k}}{\Gamma(k\theta+1)}\right)
			\|w^{1}-w^{0}\|^{p}_{\mathcal{H}^{\phi,\gamma+2}_{2}(T)}.
		\end{align*}
		Taking $\varepsilon<\frac{1}{8}$ and noting that the above maximum is finite, this implies that $\{w^{i}\}$ is a Cauchy sequence in $\mathcal{H}^{\phi,\gamma+2}_{p}(T)$. Taking $i\rightarrow\infty$ in \eqref{solution sequence}, we obtain that $w$ is a solution to \eqref{TSFSPDE} in the sense of Definition \ref{solution define}.
		
		Finally, we verify the a priori estimate. Note that $(w-w^{0})(0,\cdot)=0$, and combining with Lemma \ref{fuzhuyinli}, for each $t\leq T$,
		\begin{align*}
			\|w\|^{p}_{\mathcal{H}^{\phi,\gamma+2}_{p}(t)}&\leq\|w-w^{0}\|^{p}_{\mathcal{H}^{\phi,\gamma+2}_{p}(T)}
			+\|w^{0}\|^{p}_{\mathcal{H}^{\phi,\gamma+2}_{p}(t)}\\
			&\leq\|g(w)-g(0)\|^{p}_{\mathscr{H}^{\phi,\gamma}_{p}(t)}+\|f(w)-f(0)\|^{p}_{\mathscr{H}^{\phi,\gamma+\delta_{0}}_{p}(t,l_{2})}\\
			&\quad+\|h(w)-h(0)\|^{p}_{\mathscr{H}^{\phi,\gamma+\delta_{1}}_{p}(t,l_{2},d_{1})}+\|w^{0}\|^{p}_{\mathcal{H}^{\phi,\gamma+2}_{p}(t)}\\
			&\leq\varepsilon^{p}\|w\|^{p}_{\mathcal{H}^{\phi,\gamma+2}_{p}(t)}+C(\varepsilon)\|w\|^{p}_{\mathcal{H}^{\phi,\gamma}_{p}(t)}
			+\|w^{0}\|^{p}_{\mathcal{H}^{\phi,\gamma+2}_{p}(t)}.
		\end{align*}
		Taking $\varepsilon=\frac{1}{2}$, we derive
		\begin{align*}
			&\|w\|^{p}_{\mathcal{H}^{\phi,\gamma+2}_{p}(t)}\\
			&\lesssim_{T}\|w^{0}\|^{p}_{\mathcal{H}^{\phi,\gamma+2}_{p}(t)}
			+\|w-w^{0}\|^{p}_{\mathscr{H}^{\phi,\gamma}_{p}(t)}\\
			&\lesssim_{T}\|w^{0}\|^{p}_{\mathcal{H}^{\phi,\gamma+2}_{p}(t)}+\int_{0}^{t}(t-s)^{\theta-1}
			\big(\|\phi(\Delta)(w-w^{0})\|^{p}_{\mathscr{H}^{\phi,\gamma}_{p}(s)}\\
			&\quad+\|g(w)-g(w^{0})\|^{p}_{\mathscr{H}^{\phi,\gamma}_{p}(s)}
			+\|h(w)-h(w^{0})\|^{p}_{\mathscr{H}^{\phi,\gamma}_{p}(s,l_{2})}\\
			&\quad+\|f(w)-f(w^{0})\|^{p}_{\mathscr{H}^{\phi,\gamma}_{p}(s,l_{2},d_{1})}\big)\,ds\\
			&\lesssim_{T}\|w^{0}\|^{p}_{\mathcal{H}^{\phi,\gamma+2}_{p}(t)}
			+\int_{0}^{t}(t-s)^{\theta-1}\|w\|^{p}_{\mathcal{H}^{\phi,\gamma+2}_{p}(s)}\,ds.
		\end{align*}
		Applying Gronwall's inequality and noting that
		\[
		\|w^{0}\|_{\mathcal{H}^{\phi,\gamma+2}_{p}(T)}^{p}\leq C\bigg(\textbf{I}_{\alpha p>1}\|w_{0}\|^{p}_{\mathbb{B}^{\phi,\gamma+2-\frac{2}{\alpha p}}_{p,p}}+\|g(0)\|^{p}_{\mathscr{H}^{\phi,\gamma}_{p}(T)}+\|h(0)\|^{p}_{\mathscr{H}_{p}^{\phi,\gamma+\delta_{0}}(T,l_{2})}
		+\|f(0)\|^{p}_{\mathscr{H}_{p}^{\phi,\gamma+\delta_{1}}(T,l_{2},d_{1})}\bigg),
		\]
		we complete the proof.
	\end{proof}
Next, we apply the regularity result of Theorem~\ref{main theorem 1} to the NLSPDE \eqref{TSFSPDE1} on $\mathbb{R}^{d}$, that is the model driven by L\'{e}vy space-time white noise:
\begin{align*}
	\partial_t^\alpha w = \phi(\Delta)w +g(w)+\partial^{\sigma_{2}-1}_{t}\eta(w)\mathcal{\dot{Z}},\quad t>0,x\in\mathbb{R}^{d};
	w(0) =w_{0},~x\in\mathbb{R}^{d}.
\end{align*}
To apply the regularity result established in Theorem~\ref{main theorem 1}, we consider here that $\mathcal{Z}_{t,x}(\omega)$ is a cylindrical Wiener process. That is,~for an orthonormal basis $\{e_{k}(x)\}_{k \geq 0}$ of $L_{2}(\mathbb{R}^{d})$, $\mathcal{Z}_{t}$ admits the decomposition
$$
\mathcal{Z}_{t} = \sum_{k=1}^{\infty} \langle \mathcal{Z}_{t}, e_{k} \rangle e_{k}(x),
$$
where $\{\langle \mathcal{Z}_{t}, e_{k} \rangle\}_{k \geq 1}$ is a sequence of independent real-valued Wiener processes.

Consequently, for a function $X(s,x) = \xi(x) \mathbf{I}_{(\tau,\varsigma]}(t)$ with $\xi \in C^{\infty}_{0}(\mathbb{R}^{d})$ and $\tau, \varsigma$ being bounded stopping times, we employ the Walsh stochastic integral to obtain
\begin{align*}
	\int_{0}^{t} \int_{\mathbb{R}^{d}} X(s,x) \, dZ_{s} = \sum_{k=1}^{\infty} \int_{0}^{t} \int_{\mathbb{R}^{d}} X(s,x) e_{k}(x) \, d\mathcal{Z}^{k}_{s},
\end{align*}
where $\mathcal{Z}^{k}_{t} = \langle \mathcal{Z}_{t}, e_{k} \rangle$.~Furthermore, it is worth noting that for a more general time-space L\'{e}vy process $\mathcal{Z}_{t,x}(\omega)$, we cannot ascertain the mutual independence of $\{\langle \mathcal{Z}_{t}, e_{k} \rangle\}_{k \geq 1}$; only their uncorrelatedness can be established \cite{Applebaum}.

Thus, the L\'{e}vy space-time white noise model \eqref{TSFSPDE1} is transformed into the following stochastic model with initial value $w_{0}$:
\begin{align}\label{phi xingshibiaoda}
	\partial_t^\alpha w = \phi(\Delta) w + g(w) + \sum_{k=1}^{\infty} \partial_t^{\sigma_{2}} \int_{0}^{t} \eta(w) e_{k}(x) \, d\mathcal{Z}^{k}_{s}.
\end{align}

We apply the result of Theorem~\ref{main theorem 1} to \eqref{phi xingshibiaoda}. First, we present the following lemma, see \cite[Lemma 5.2]{Kim3}.
	
	\begin{lemma}\label{FZLemma}
		Let $k_{0}\in(\frac{d}{2\kappa_{0}},\frac{d}{\kappa_{0}})$, $2\leq 2r\leq p$, and $2r<d/(d-k_{0}\kappa_{0})$. Suppose the function $J(\cdot,\cdot)$ satisfies
		\[
		|J(x,u)-J(x,w)|\lesssim \xi(x)|u-w|,\quad \forall x\in\mathbb{R}^{d}, u,w\in\mathbb{R}.
		\]
		Let $\mathcal{J}^{k}(x,w)=J(x,w)e_{k}(x)$. Then we have
		\[
		\|\mathcal{J}(u)-\mathcal{J}(w)\|_{H_{p}^{\phi,-k_{0}}(l_{2})}\lesssim\|\xi\|_{L_{\frac{2r}{r-1}}}\|u-w\|_{L_{p}}.
		\]
		In particular, for $\xi\in L_{\infty}$ and $r=1$, we have
		\[
		\|\mathcal{J}(u)-\mathcal{J}(w)\|_{H_{p}^{\phi,-k_{0}}(l_{2})}\lesssim\|u-w\|_{L_{p}}.
		\]
	\end{lemma}
	
	\begin{theorem}\label{main theorem 2}
		Let $\kappa_{0}\in(\frac{1}{4},1]$, and denote that $f^{k}(t,x,w):=\eta(w)e_{k}$. Assume the functions $g$ and $\eta$ satisfy the following conditions:
		\begin{align*}
			|g(t,x,u)-g(t,x,w)|&\lesssim |u-w|,\\
			|\eta(t,x,u)-\eta(t,x,w)|&\lesssim \xi(t,x)|u-w|,
		\end{align*}
		where $\xi$ is a function of $(\omega,t,x)$. If the following conditions hold:
		\[
		\|\eta(0)\|_{\mathscr{L}_{p}(T)}+\|g(0)\|_{\mathscr{H}^{\phi,-k_{0}-\delta_{1}}_{p}(T)}+\sup_{t,\omega}\|\xi\|_{L_{2s}}<\infty,
		\]
		where the constants $k_{0}$ and $s$ satisfy
		\begin{align}\label{constant satisfy condition}
			\frac{d}{2\kappa_{0}}<k_{0}<\frac{d}{\kappa_{0}}\wedge\left(2-\frac{(2\sigma_{2}-2/p)_{+}}{\alpha}\right),\quad
			\frac{d}{2k_{0}\kappa_{0}-d}<s,
		\end{align}
		then Equation \eqref{TSFSPDE} admits a unique solution $w$ in $\mathcal{H}^{\phi,2-k_{0}-\delta_{1}}_{p}(T)$ satisfying the estimate
		\[
		\|w\|_{\mathcal{H}^{\phi,2-k_{0}-\delta_{1}}_{p}(T)}\lesssim_{T}\left(\textbf{I}_{\alpha p>1}\|w_{0}\|_{\mathbb{B}^{\phi,2-k_{0}-\delta_{1}-\frac{2}{\alpha p}}_{p,p}}+\|g(0)\|_{\mathscr{H}^{\phi,-k_{0}-\delta_{1}}_{p}(T)}
		+\|\eta(0)\|_{\mathscr{L}_{p}(T)}\right).
		\]
	\end{theorem}
	
	\begin{proof}
 Denote that $f(t,x,w)e_{k}(x):=F^{k}(t,x,w)$, we need to verify that Assumption \ref{assumption of nonlinear function} holds for $\gamma=-k_{0}-\delta_{1}$. Condition \eqref{constant satisfy condition} implies that $\gamma+2>0$, hence
		\begin{align*}
			\|g(t,x,w)-g(t,x,v)\|_{H^{\phi,\gamma}_{p}}&\lesssim\|g(t,x,w)-g(t,x,v)\|_{L_{p}}\\
			&\lesssim\|w-u\|_{L_{p}}\\
			&\lesssim \varepsilon\|w-u\|_{H^{\phi,\gamma+2}_{p}}+N(\varepsilon)\|w-u\|_{H^{\phi,\gamma}_{p}}.
		\end{align*}
		On the other hand, taking $s=\frac{r}{r-1}$, Condition \eqref{constant satisfy condition} implies that $2r<d/(d-k_{0}\kappa_{0})$. Using Lemma \ref{FZLemma}, we obtain
		\begin{align*}
			\|F(t,x,w)-F(t,x,v)\|_{H^{\phi,\gamma+\delta_{1}}_{p}(l_{2})}&\lesssim\|F(t,x,w)-F(t,x,v)\|_{H_{p}^{\phi,-k_{0}}(l_{2})}\\
			&\lesssim\|\xi\|_{L_{2s}}\|u-w\|_{L_{p}}\\
			&\lesssim\varepsilon\|w-u\|_{H^{\phi,\gamma+2}_{p}}+N(\varepsilon)\|w-u\|_{H^{\phi,\gamma}_{p}}.
		\end{align*}
		Therefore, by Theorem \ref{main theorem 1} and Lemma \ref{FZLemma}, there exists a unique $w\in\mathcal{H}^{\phi,\gamma+2}_{p}(T)$ satisfying the estimate
		\begin{align*}
			\|w\|_{\mathcal{H}^{\phi,2-k_{0}-\delta_{1}}_{p}(T)}&\leq C\left(\textbf{I}_{\alpha p>1}\|w_{0}\|_{\mathbb{B}^{\phi,2-k_{0}-\delta_{1}-\frac{2}{\alpha p}}_{p,p}}+\|g(0)\|_{\mathscr{H}^{\phi,-k_{0}-\delta_{1}}_{p}(T)}
			+\|F(0)\|_{\mathscr{H}_{p}^{\phi,-k_{0}}(T,l_{2})}\right)\\
			&\leq C\left(\textbf{I}_{\alpha p>1}\|w_{0}\|_{\mathbb{B}^{\phi,2-k_{0}-\delta_{1}-\frac{2}{\alpha p}}_{p,p}}+\|g(0)\|_{\mathscr{H}^{\phi,-k_{0}-\delta_{1}}_{p}(T)}
			+\|\eta(0)\|_{\mathscr{L}_{p}(T)}\right).
		\end{align*}
	\end{proof}
	
	\begin{remark}
		Theorem \rm\ref{main theorem 2} implies that we must have
		\[
		d<2\kappa_{0}\left(2-\frac{(2\sigma_{2}-2/p)_{+}}{\alpha}\right),\quad \sigma_{2}<\alpha\left(1-\frac{1}{4\kappa_{0}}\right)+\frac{1}{p}.
		\]
		Therefore, we can take
		\[
		d=\begin{cases}
			1,2,3, & \text{if } \sigma_{2}<\alpha\left(1-\frac{3}{4\kappa_{0}}+\frac{1}{p}\right),\\
			1, & \text{if } \alpha\left(1-\frac{1}{2\kappa_{0}}+\frac{1}{p}\right)<\sigma_{2}<\alpha\left(1-\frac{1}{4\kappa_{0}}\right)+\frac{1}{p}.
		\end{cases}
		\]
	\end{remark}
	\section{Local mild solution}
In this section,~we consider the case where $\mathcal{Z}_{t,x}$ is a general L\'{e}vy space-time white noise and establish the well-posedness of its mild solution in $L_{p}(\mathbb{R}^{d})$ ($1 \leq p \leq 2$) for NLSPDE \eqref{TSFSPDE1}. Recall the L\'{e}vy-It\^{o} decomposition, there exist $g_{1}, g_{2}: \mathbb{R}_{+} \times \mathbb{R}^{d} \times A \rightarrow \mathbb{R}$, a set $N_{0} \in \mathcal{B}(A)$ with $\mu(A \setminus N_{0}) < \infty$, such that
\begin{align*}
	\mathcal{Z}_{t,x}(\omega) = W_{t,x}(\omega) + \int_{N_{0}} g_{1}(t,x,\xi,\omega) \, \tilde{\Pi}(d\xi,\omega) + \int_{A \setminus N_{0}} g_{2}(t,x,\xi,\omega) \, \Pi(d\xi,\omega).
\end{align*}

Note that the mild solution of \eqref{TSFSPDE1} can be represented by the following integral equation:
\begin{align*}
	w(t,x) &= \mathcal{S}(t) \star w_{0}(x) + \int_{0}^{t} \mathcal{S}_{\alpha,1}(t-s) \star g(s,x,w(s,x)) \, ds \\
	&\quad + \int_{0}^{t} \int_{\mathbb{R}^{d}} \mathcal{S}_{\alpha,\sigma_{2}}(t-s, x-y) \eta(s,y,w(s,y)) \, W(ds,dy) \\
	&\quad + \int_{0}^{t} \int_{\mathbb{R}^{d}} \int_{N_{0}} \mathcal{S}_{\alpha,\sigma_{2}}(t-s, x-y) \eta(s,y,w(s,y)) g_{1}(s,y,\xi) \, \tilde{\Pi}(ds,dy,d\xi) \\
	&\quad + \int_{0}^{t} \int_{\mathbb{R}^{d}} \int_{A \setminus N_{0}} \mathcal{S}_{\alpha,\sigma_{2}}(t-s, x-y) \eta(s,y,w(s,y)) g_{2}(s,y,\xi) \, \Pi(ds,dy,d\xi),
\end{align*}
where we note that
\begin{align*}
	&\int_{0}^{t} \int_{\mathbb{R}^{d}} \int_{A \setminus N_{0}} \mathcal{S}_{\alpha,\sigma_{2}}(t-s, x-y) \eta(s,y,w(s,y)) g_{2}(s,y,\xi) \, \Pi(ds,dy,d\xi) \\
	&= \int_{0}^{t} \int_{\mathbb{R}^{d}} \int_{A \setminus N_{0}} \mathcal{S}_{\alpha,\sigma_{2}}(t-s, x-y) \eta(s,y,w(s,y)) g_{2}(s,y,\xi) \, \tilde{\Pi}(ds,dy,d\xi) \\
	&\quad + \int_{0}^{t} \int_{\mathbb{R}^{d}} \int_{A \setminus N_{0}} \mathcal{S}_{\alpha,\sigma_{2}}(t-s, x-y) \eta(s,y,w(s,y)) g_{2}(s,y,\xi) \, ds \, dy \, \mu(d\xi).
\end{align*}
Therefore, without loss of generality, we can introduce the following assumption:
\begin{assumption}\label{levy zao sheng jiashe}
	For each $\omega \in \Omega$, there exists a measurable function $h: \mathbb{R}_{+} \times \mathbb{R}^{d} \times A \rightarrow \mathbb{R}$ such that the L\'{e}vy space-time white noise $\mathcal{Z}_{t,x}$ admits the decomposition
	$$
	\mathcal{Z}_{t,x}(\omega) = W_{t,x}(\omega) + \int_{A} h(t,x,\xi,\omega) \, \tilde{\Pi}(d\xi,\omega).
	$$
\end{assumption}

Consequently, under Assumption \ref{levy zao sheng jiashe}, the mild solution of \eqref{TSFSPDE1} can be expressed by the following integral equation:
\begin{align}\label{phi-shikongbaizaosheng mild solution}
	w(t,x) &= \mathcal{S}(t) \star w_{0}(x) + \int_{0}^{t} \mathcal{S}_{\alpha,1}(t-s) \star g(s,x,w(s,x)) \, ds \notag \\
	&\quad + \int_{0}^{t} \int_{\mathbb{R}^{d}} \mathcal{S}_{\alpha,\sigma_{2}}(t-s, x-y) \eta(s,y,w(s,y)) \, W(ds,dy) \notag \\
	&\quad + \int_{0}^{t} \int_{\mathbb{R}^{d}} \int_{A} \mathcal{S}_{\alpha,\sigma_{2}}(t-s, x-y) \eta(s,y,w(s,y)) h(s,y,\xi) \, \tilde{\Pi}(ds,dy,d\xi).
\end{align}
Furthermore, if $\mathcal{Z}_{t,x}$ is a pure-jump L\'{e}vy space-time white noise, then the mild solution of \eqref{TSFSPDE1} is given by
\begin{align}\label{jump phi-shikongbaizaosheng mild solution}
	w(t,x) &= \mathcal{S}(t) \star w_{0}(x) + \int_{0}^{t} \mathcal{S}_{\alpha,1}(t-s) \star g(s,x,w(s,x)) \, ds \notag \\
	&\quad + \int_{0}^{t} \int_{\mathbb{R}^{d}} \int_{A} \mathcal{S}_{\alpha,\sigma_{2}}(t-s, x-y) \eta(s,y,w(s,y)) h(s,y,\xi) \, \tilde{\Pi}(ds,dy,d\xi).
\end{align}

\begin{definition}[Local mild solution]
	Let $T>0$, and consider an $\mathcal{F}_{t}$-adapted stochastic process $w:[0,T]\times\mathbb{R}^{d}\rightarrow\mathbb{R}$ which is c\`{a}dl\`{a}g in $t$. If there exists an $\mathcal{F}_{t}$-stopping time $\upsilon:\Omega\rightarrow [0,T]$ such that $\{w(t,x)\}_{t\leq \upsilon}$ satisfies \eqref{phi-shikongbaizaosheng mild solution}~{\rm(resp.\eqref{jump phi-shikongbaizaosheng mild solution})}, then we say $w$ is a local mild solution of \eqref{TSFSPDE1} driven by L\'{e}vy time-space white noise ~{\rm(resp.~pure jump L\'{e}vy noise)}. Moreover, if for any other mild solution $v$ with stopping time $\tilde{\upsilon}$, we have $w(t,x)=v(t,x)$ almost surely for all $t\in [0,\upsilon\wedge\tilde{\upsilon}]\times\mathbb{R}^{d}$, then we say the mild solution is unique.
\end{definition}

The following lemma is crucial in establishing the mild solution.
\begin{lemma}[\cite{Wu}]\label{gujibudengshi}
	Let $1\leq p\leq 2$, $\phi:[0,\infty)\times\mathbb{R}^{d}\times \Omega\rightarrow\mathbb{R}$ is a $\mathcal{F}_{t}$-adapted function, if
	$$
	\int_{0}^{t}\int_{\mathbb{R}^{d}}\int_{A}\mathbb{E}[|\phi(s,x,\xi)|^{p}]\,ds\,dx\,\mu(\,d\xi)<\infty,
	$$
	then
	$$
	\int_{0}^{t}\int_{\mathbb{R}^{d}}\int_{A}\phi(s,x,\xi)\tilde{\Pi}(\,ds,\,dx,\,d\xi)
	$$
	is well-defined in $L_{p}(\Omega,\mathcal{F},\mathbb{P})$, and the following hohd:
	$$
	\mathbb{E}\bigg[\bigg|\int_{0}^{t}\int_{\mathbb{R}^{d}}\int_{A}\phi(s,x,\xi)\tilde{\Pi}(\,ds,\,dx,\,d\xi)\bigg|^{p}\bigg]\lesssim
	\int_{0}^{t}\int_{\mathbb{R}^{d}}\int_{A}\mathbb{E}[|\phi(s,x,\xi)|^{p}]\,ds\,dx\,\mu(\,d\xi).
	$$
\end{lemma}
\begin{lemma}\label{keyestimateformildsolution}
	Let $\alpha d<4\delta$ or $1\leq p<\frac{\alpha d}{\alpha d-4\kappa_{0}}$,$\beta\in\mathbb{R}$,there exist constant $C=C(\alpha,\beta,\kappa_{0},p)$ such that
	\begin{align*}
		\int_{\mathbb{R}^{d}}\big|\mathcal{S}_{\alpha,\beta}(t,x)\big|^{p}\,dx\lesssim t^{(\alpha-\beta)p}\big(\phi^{-1}(t^{-\alpha})\big)^{\frac{d}{2}(p-1)}.
	\end{align*}
\end{lemma}
\begin{proof}
	Note that
	\begin{align*}
		\int_{\mathbb{R}^{d}}\big|\mathcal{S}_{\alpha,\beta}(t,x)\big|^{p}\,dx=\int_{|x|\geq (\phi^{-1}(t^{-\alpha}))^{-\frac{1}{2}}}
		\big|\mathcal{S}_{\alpha,\beta}(t,x)\big|^{p}\,dx+\int_{|x|<(\phi^{-1}(t^{-\alpha}))^{-\frac{1}{2}}}
		\big|\mathcal{S}_{\alpha,\beta}(t,x)\big|^{p}\,dx
	\end{align*}
	Basic Lemma \ref{Some proposition of S} and \eqref{lower scailing condition}, we derive
	\begin{align*}
		\int_{|x|\geq \big(\phi^{-1}(t^{-\alpha})\big)^{-\frac{1}{2}}}
		\big|\mathcal{S}_{\alpha,\beta}(t,x)\big|^{p}\,dx&\leq\int_{|x|\geq \big(\phi^{-1}(t^{-\alpha})\big)^{-\frac{1}{2}}}\big|t^{2\alpha-\beta}\frac{\phi(|x|^{-2})}{|x|^{d}}\big|^{p}\,dx\\
		&\leq\int_{(\phi^{-1}(t^{-\alpha}))^{-\frac{1}{2}}}^{\infty}t^{(2\alpha-\beta)p}\frac{|\phi(r^{-2})|^{p}}{r^{(p-1)d+1}}\,dr\\
		&\lesssim\int_{(\phi^{-1}(t^{-\alpha}))^{-\frac{1}{2}}}^{\infty}t^{(\alpha-\beta)p}
		\frac{(\phi^{-1}(t^{-\alpha}))^{-p}}{r^{1+2p+(p-1)d}}\,dr\\
		&\lesssim t^{(\alpha-\beta)p}\big(\phi^{-1}(t^{-\alpha})\big)^{\frac{(p-1)d}{2}}.
	\end{align*}
	By the Minkowski inequality and \eqref{lower scailing condition}, we derive
	\begin{align*}
		\int_{|x|<(\phi^{-1}(t^{-\alpha}))^{-\frac{1}{2}}}
		\big|\mathcal{S}_{\alpha,\beta}(t,x)\big|^{p}\,dx&\lesssim\int_{|x|<(\phi^{-1}(t^{-\alpha}))^{-\frac{1}{2}}}
		\bigg|\int_{(\phi(|x|^{-2}))^{-1}}^{2t^{\alpha}}(\phi^{-1}(r^{-1}))^{\frac{d}{2}}rt^{-\alpha-\beta}\,dr\bigg|^{p}\,dx\\
		&\lesssim\bigg[\int_{0}^{2t^{\alpha}}\bigg(\int_{(\phi(|x|^{-1}))^{-1}\leq r}\big|(\phi^{-1}(r^{-1}))^{\frac{d}{2}}rt^{-\alpha-\beta}\big|^{p}\,dx\bigg)^{\frac{1}{p}}\,dr\bigg]^{p}\\
		&\lesssim\bigg(\int_{0}^{2t^{\alpha}}(\phi^{-1}(r^{-1}))^{\frac{d}{2}(\frac{p-1}{p})}rt^{-\alpha-\beta}\,dr\bigg)^{p}\\
		&\lesssim t^{-(\alpha+\beta)p}\bigg(\int_{0}^{2t^{\alpha}}(\phi^{-1}(t^{-\alpha}))^{\frac{d}{2}(\frac{p-1}{p})}t^{\frac{\alpha d}{2\kappa_{0}}(\frac{p-1}{p})}r^{1-\frac{\alpha d}{2\kappa_{0}}(\frac{p-1}{p})}\,dr\bigg)^{p}\\
		&\lesssim t^{(\alpha-\beta)p}\big(\phi^{-1}(t^{-\alpha})\big)^{\frac{(p-1)d}{2}}.
	\end{align*}
\end{proof}
\begin{theorem}\label{local mild solution theorem}
	Let $p\in [1,2]$, $T>0$, $\mathcal{Z}_{t,x}$ is a pure jump L\'{e}vy time-space white noise in NLSPDE \eqref{TSFSPDE1}, and assume the following condition hold:
	\[
	(\alpha-\sigma_{2})p+1>\frac{\alpha d}{2\kappa_{0}}(p-1).
	\]
We denote $\tilde{f}=\eta h$ and suppose there exist functions $\theta_{1},\theta_{2},\theta_{3}\in L_{p}(\mathbb{R}^{d})$ such that for any $(t,x,\xi)\in [0,T]\times\mathbb{R}^{d}\times A$, $z_{1},z_{2}\in\mathbb{R}$, the following estimates hold:
	\[
	|g(t,x,z)|\lesssim \theta_{1}(x)+|z|,\quad \int_{A}|\tilde{f}(t,x,\xi,z)|^{p}\,\mu(d\xi)\lesssim |\theta_{2}(x)|^{p}+|z|^{p},
	\]
	and
	\begin{align*}
		|g(t,x,z_{1})-g(t,x,z_{2})|&\lesssim (\theta_{3}(x)+|z_{1}|^{p-1}+|z_{2}|^{p-1})|z_{1}-z_{2}|,\\
		\int_{A}|\tilde{f}(t,x,\xi,z_{1})-\tilde{f}(t,x,\xi,z_{2})|^{p}\,dt\,dx\,\mu(d\xi)&\lesssim |z_{1}-z_{2}|^{p}.
	\end{align*}
	Let the $\mathcal{F}_{0}$-adapted process $w_{0}$ satisfy $\mathbb{E}[\|w_{0}\|_{L_{p}}^{p}]<\infty$, then equation \eqref{TSFSPDE} admits a unique local mild solution $w$ on $[0,T]\times\mathbb{R}^{d}$ which has a predictable modification, and satisfies
	\[
	\mathbb{E}[\|w(t\wedge\upsilon,\cdot)\|_{p}^{p}]<\infty.
	\]
\end{theorem}

\begin{proof}
	For fixed $T>0$, $p\in[1,2]$, we introduce the Banach space $B_{T,p}$ consisting of $\mathcal{F}_{t}$-adapted stochastic functions $w(t,x)$ satisfying
	\[
	\|w\|_{B_{T,p}}=\sup_{t\in[0,T]}\mathbb{E}[\|w(t)\|_{L_{p}}^{p}]^{\frac{1}{p}}<\infty.
	\]
	For any fixed $K\in\mathbb{N}_{+}$, we define the mapping $\lambda_{K}:L_{p}(\mathbb{R}^{d})\rightarrow L_{p}(\mathbb{R}^{d})$ by
	\[
	\lambda_{K}w_{1}(x)=
	\begin{cases}
		w_{1}(x), & \|w_{1}\|_{p}\leq K,\\
		\frac{Kw_{1}(x)}{\|w_{1}\|_{p}}, & \|w_{1}\|_{p}> K.
	\end{cases}
	\]
	It is easy to see that $\|\lambda_{K}w_{1}\|_{p}\leq K$, and $\|\lambda_{K}w_{1}-\lambda_{K}w_{2}\|_{p}\leq \|w_{1}-w_{2}\|_{p}$. We define the following operator associated with the stochastically truncated function $\lambda_{K}w(t,x)$:
	\begin{align*}
		\mathcal{T}w(t,x)&=\mathcal{S}_{\alpha,\alpha}(t)\star w_{0}(x)+\int_{0}^{t}\mathcal{S}_{\alpha,1}(t-s)\star g(s,x,\lambda_{K}w(s,x))\,ds\\
		&\quad+\int_{0}^{t}\int_{\mathbb{R}^{d}}\int_{A}
		\mathcal{S}_{\alpha,\sigma_{2}}(t-s,x-y)\tilde{f}(s,y,\xi,\lambda_{K}w(s,y))M(ds,dy,d\xi)\\
		&\triangleq \mathcal{T}_{1}w(t,x)+\mathcal{T}_{2}w(t,x)+\mathcal{T}_{3}w(t,x).
	\end{align*}
	
	First, we verify that the operator $\mathcal{T}$ maps $B_{T,p}$ into $B_{T,p}$.
	
	Using Lemma \ref{Some proposition of S} and Young's inequality, we derive
	\[
	\|\mathcal{T}_{1}w\|_{L_{p}}\lesssim\|\mathcal{S}_{\alpha,\alpha}(t)\star w_{0}\|_{p}\lesssim\|w_{0}\|_{p},
	\]
	and
	\begin{align*}
		\|\mathcal{T}_{2}w\|_{L_{p}}&\lesssim\left\|\int_{0}^{t}\int_{\mathbb{R}^{d}}\mathcal{S}_{\alpha,1}(t-s,x-y) g(s,y,\lambda_{K}w(s,y))\,ds\,dy\right\|_{L_{p}}\\
		&\lesssim\int_{0}^{t}(t-s)^{\alpha-1}\,ds(\|\theta_{1}\|_{p}+K)\lesssim T^{\alpha}(\|\theta_{1}\|_{p}+K)<\infty.
	\end{align*}
	
	Using Lemma \ref{gujibudengshi} and \eqref{lower scailing condition}, we derive
	\begin{align*}
		\mathbb{E}[\|\mathcal{T}_{3}w\|_{L_{p}}^{p}]&\lesssim\int_{\mathbb{R}^{d}}\mathbb{E}\left|\int_{0}^{t}\int_{\mathbb{R}^{d}}\int_{A}
		\mathcal{S}_{\alpha,\sigma_{2}}(t-s,x-y)\tilde{f}(s,y,\xi,\lambda_{K}w(s,y))M(ds,dy,d\xi)\right|^{p}\,dx\\
		&\lesssim\mathbb{E}\int_{\mathbb{R}^{d}}\int_{0}^{t}\int_{\mathbb{R}^{d}}\int_{A}
		\left|\mathcal{S}_{\alpha,\sigma_{2}}(t-s,x-y)\tilde{f}(s,y,\xi,\lambda_{K}w(s,y))\right|^{p}\,ds\,dy\,\mu(d\xi)\,dx\\
		&\lesssim\int_{\mathbb{R}^{d}}\int_{0}^{t}\int_{\mathbb{R}^{d}}|\mathcal{S}_{\alpha,\sigma_{2}}(t-s,x-y)|^{p}
		(|\theta_{2}(y)|^{p}+|\lambda_{K}w(s,y)|^{p})\,ds\,dy\,dx\\
		&\lesssim\int_{0}^{t}\int_{\mathbb{R}^{d}}|\mathcal{S}_{\alpha,\sigma_{2}}(t-s,x)|^{p}\,dx\,ds
		(\|\theta_{2}\|^{p}_{L_{p}}+K^{p})\\
		&\lesssim\int_{0}^{t}(t-s)^{(\alpha-\sigma_{2})p}(\phi^{-1}((t-s)^{-\alpha}))^{\frac{d(p-1)}{2}}\,ds
		(\|\theta_{2}\|^{p}_{L_{p}}+K^{p})\\
		&\lesssim(\phi^{-1}(T^{-\alpha})T^{\frac{\alpha}{\kappa_{0}}})^{\frac{(p-1)d}{2}}
		\int_{0}^{t}(t-s)^{(\alpha-\sigma_{2})p-\frac{\alpha d}{2\kappa_{0}}(p-1)}\,ds(\|\theta_{2}\|^{p}_{L_{p}}+K^{p})\\
		&\lesssim_{T}(\|\theta_{2}\|^{p}_{L_{p}}+K^{p})<\infty.
	\end{align*}
	
	Combining the estimates for $\mathcal{T}_{1}w$, $\mathcal{T}_{2}w$, and $\mathcal{T}_{3}w$, we conclude that the operator $\mathcal{T}$ maps $B_{T,p}$ into $B_{T,p}$.
	
	Next, for $\vartheta>0$, we introduce the Banach space $B_{\vartheta,p}$ consisting of $\mathcal{F}_{t}$-adapted stochastic functions $w(t,x)$ satisfying
	\[
	\|w\|^{p}_{B_{\vartheta,p}}=\sup_{t\in[0,T]}e^{-\vartheta t}\mathbb{E}[\|w(t)\|_{L_{p}}^{p}]<\infty.
	\]
	It is easy to see that the norm $\|w\|_{B_{\vartheta,p}}$ is equivalent to $\|w\|_{B_{T,p}}$ for fixed $\vartheta>0$. We verify that the operator $\mathcal{T}$ is a contraction on $B_{\vartheta,p}$ for sufficiently large $\vartheta>0$. Following a similar procedure as above, we can verify that $\mathcal{T}$ maps $B_{\vartheta,p}$ into itself. Moreover, for any $w_{1},w_{2}\in B_{\vartheta,p}$, by Jensen's inequality, we derive
	\begin{align*}
		&\sup_{t\in[0,T]}e^{-\vartheta t}\mathbb{E}\|\mathcal{T}_{2}w_{1}-\mathcal{T}_{2}w_{2}\|^{p}_{L_{p}}\\
		&\lesssim \sup_{t\in[0,T]}e^{-\vartheta t}\mathbb{E}\left[\left\|\int_{0}^{t}\int_{\mathbb{R}^{d}}
		\mathcal{S}_{\alpha,1}(t-s,x-y)(g(s,y,\lambda_{K}w_{1}(s,y))-g(s,y,\lambda_{K}w_{2}(s,y)))\,ds\,dy\right\|^{p}_{L_{p}}\right]\\
		&\lesssim\sup_{t\in[0,T]}e^{-\vartheta t}\mathbb{E}\left[\int_{0}^{t}(t-s)^{\alpha-1}\|(\theta_{3}(y)+|\lambda_{K}w_{1}|^{p-1}+|\lambda_{K}w_{2}|^{p-1})
		|\lambda_{K}w_{1}-\lambda_{K}w_{2}|\|_{L_{p}}\,ds\right]^{p}\\
		&\lesssim\sup_{t\in[0,T]}e^{-\vartheta t}\mathbb{E}\left[\int_{0}^{t}(t-s)^{\alpha-1}(\|\lambda_{K}w_{1}-\lambda_{K}w_{2}\|_{L_{p}}
		(\|\theta_{3}\|_{L_{p}}+\|\lambda_{K}w_{1}\|^{p-1}_{L_{p}}+\|\lambda_{K}w_{2}\|^{p-1}_{L_{p}}))\,ds\right]^{p}\\
		&\lesssim\sup_{t\in[0,T]}\int_{0}^{t}e^{-\vartheta(t-s)}(t-s)^{(\alpha-1)p}e^{-\vartheta s}\mathbb{E}\|\lambda_{K}w_{1}-\lambda_{K}w_{2}\|^{p}_{L_{p}}\,ds
		(\|\theta_{3}\|_{L_{p}}+2K^{p-1})^{p}\\
		&\leq\frac{1}{2}\sup_{t\in[0,T]}e^{-\vartheta t}\mathbb{E}\|w_{1}-w_{2}\|^{p}_{L_{p}},\quad \text{for sufficiently large }\vartheta>0,
	\end{align*}
	and
	\begin{align*}
		&\sup_{t\in[0,T]}e^{-\vartheta t}\mathbb{E}\|\mathcal{T}_{3}w_{1}-\mathcal{T}_{3}w_{2}\|^{p}_{L_{p}}\\
		&\lesssim\sup_{t\in[0,T]}e^{-\vartheta t}\int_{\mathbb{R}^{d}}\mathbb{E}\left|\int_{0}^{t}\int_{\mathbb{R}^{d}}\int_{A}
		\mathcal{S}_{\alpha,\sigma_{2}}(t-s,x-y)(\tilde{f}(s,y,\xi,\lambda_{K}w_{1}(s,y))\right.\\
		&\qquad\qquad\left.-\tilde{f}(s,y,\xi,\lambda_{K}w_{2}(s,y)))M(ds,dy,d\xi)\right|^{p}\,dx\\
		&\lesssim\sup_{t\in[0,T]}e^{-\vartheta t}\mathbb{E}\int_{\mathbb{R}^{d}}\int_{0}^{t}\int_{\mathbb{R}^{d}}\int_{A}
		\left|\mathcal{S}_{\alpha,\sigma_{2}}(t-s,x-y)(\tilde{f}(s,y,\xi,\lambda_{K}w_{1}(s,y))\right.\\
		&\qquad\qquad\left.-\tilde{f}(s,y,\xi,\lambda_{K}w_{2}(s,y)))\right|^{p}\,ds\,dy\,\mu(d\xi)\,dx\\
		&\lesssim(\phi^{-1}(T^{-\alpha})T^{\frac{\alpha}{\kappa_{0}}})^{\frac{(p-1)d}{2}}\sup_{t\in[0,T]}
		e^{-\vartheta t}\int_{0}^{t}(t-s)^{(\alpha-\sigma_{2})p-\frac{\alpha d}{2\kappa_{0}}(p-1)}
		\|\lambda_{K}w_{1}-\lambda_{K}w_{2}\|^{p}_{L_{p}}\,ds\\
		&\leq\frac{1}{2}\sup_{t\in[0,T]}e^{-\vartheta t}\mathbb{E}\|w_{1}-w_{2}\|^{p}_{L_{p}},\quad \text{for sufficiently large }\vartheta>0.
	\end{align*}
	
	In summary, we obtain that the operator $\mathcal{T}$ is a contraction on $B_{\vartheta,p}$ for sufficiently large $\vartheta>0$. By the Banach fixed point theorem, for any fixed $\vartheta$, the operator $\mathcal{T}$ has a unique fixed point $w_{K}$ in $B_{\vartheta,p}$, which is the unique solution to the equation
	\begin{align}\label{local mild solution}
		w(t,x)&=\mathcal{S}_{\alpha,\alpha}(t)\star w_{0}(x)+\int_{0}^{t}\mathcal{S}_{\alpha,1}(t-s)\star g(s,x,\lambda_{K}w(s,x))\,ds\notag\\
		&\quad+\int_{0}^{t}\int_{\mathbb{R}^{d}}\int_{A}
		\mathcal{S}_{\alpha,\sigma_{2}}(t-s,x-y)\tilde{f}(s,y,\xi,\lambda_{K}w(s,y))M(ds,dy,d\xi).
	\end{align}
	
	Next, we construct an $\mathcal{F}_{t}$-stopping time $\upsilon_{K}$. Let
	\[
	\upsilon_{K}:=\inf\{t\in[0,T]:\|w_{K}(t)\|_{L_{p}}>K\}.
	\]
	By the monotone convergence theorem, $\upsilon=\lim_{K\rightarrow\infty}\upsilon_{K}$ exists. Noting the uniqueness of the local mild solution of Equation \eqref{TSFSPDE}, for any $N>K$, we have
	\[
	w_{N}(t,x,\cdot)=w_{K}(t,x,\cdot)\quad \text{for a.e. }t\in [0,T],x\in\mathbb{R}^{d}.
	\]
	Hence, for any $K\in\mathbb{N}_{+}$, we define
	\[
	w(t,x,\omega)=w_{K}(t,x,\omega) \quad \text{for }(t,x,\omega)\in[0,\upsilon_{K})\times\mathbb{R}^{d}\times\Omega.
	\]
	Clearly, through this definition, we obtain a local mild solution of Equation \eqref{TSFSPDE} with respect to the $\mathcal{F}_{t}$-stopping time $\upsilon$. Moreover, for any two local mild solutions $w_{1},w_{2}$ satisfying \eqref{mild solution formula}, by the definition of local mild solution, for any $K\in\mathbb{N}_{+}$, $w_{1}(t)=w_{2}(t)$ for $t\in [0,\upsilon_{K})$. Letting $K\rightarrow\infty$, we obtain that the mild solution of Equation \eqref{TSFSPDE} is unique. The condition $\mathbb{E}[\|w(t\wedge\upsilon,\cdot)\|_{p}^{p}]<\infty$ is obvious.
	
	Finally, we verify that the mild solution $w$ has a predictable modification. From \cite[Proposition 3.21]{Peszat}, any stochastically continuous $\mathcal{F}_{t}$-adapted process has a predictable modification. Thus, it suffices to verify
	\begin{align}\label{continuous estimate}
		&\lim_{t_{2}\rightarrow t_{1}}\int_{\mathbb{R}^{d}}\mathbb{E}\left[\left|\int_{0}^{t_{2}}\int_{\mathbb{R}^{d}}\int_{A}
		\mathcal{S}_{\alpha,\sigma_{2}}(t_{2}-s,x-y)
		\tilde{f}(s,y,\xi,w(s,y))M(ds,dy,d\xi)\right.\right.\notag\\
		&\qquad\qquad\left.\left.-\int_{0}^{t_{1}}\int_{\mathbb{R}^{d}}\int_{A}
		\mathcal{S}_{\alpha,\sigma_{2}}(t_{1}-s,x-y)
		\tilde{f}(s,y,\xi,w(s,y))M(ds,dy,d\xi)\right|^{p}\right]\,dx=0.
	\end{align}
	
	Note that the left-hand side of \eqref{continuous estimate} is controlled by
	\begin{align}\label{control estimate}
		&\mathbb{E}\int_{\mathbb{R}^{d}}\int_{0}^{t_{1}}\int_{\mathbb{R}^{d}}|\mathcal{S}_{\alpha,\sigma_{2}}(t_{2}-s)
		-\mathcal{S}_{\alpha,\sigma_{2}}(t_{1}-s)|^{p}(\theta_{4}(y)+|w(s,y)|^{p})\,ds\,dy\,dx\notag\\
		&\quad+\mathbb{E}\int_{\mathbb{R}^{d}}\int_{t_{1}}^{t_{2}}\int_{\mathbb{R}^{d}}|\mathcal{S}_{\alpha,\sigma_{2}}(t_{2}-s)
		|^{p}(\theta_{4}(y)+|w(s,y)|^{p})\,ds\,dy\,dx\triangleq I_{1}+I_{2}.
	\end{align}
	
	For $I_{2}$, we derive
	\begin{align*}
		I_{2}&\lesssim(\phi^{-1}(T^{-\alpha})T^{\frac{\alpha}{\kappa_{0}}})^{\frac{(p-1)d}{2}}
		\int_{t_{1}}^{t_{2}}(t_{2}-s)^{(\alpha-\sigma_{2})p-\frac{\alpha d}{2\kappa_{0}}(p-1)}
		(\|\theta_{4}\|_{L_{1}}+\sup_{s\in[0,T]}\mathbb{E}\|w(s)\|_{L_{p}}^{p})\,ds\\
		&\rightarrow 0\quad \text{as }t_{2}\rightarrow t_{1}.
	\end{align*}
	
	For $I_{1}$, we derive
	\begin{align*}
		I_{1}&\lesssim(\phi^{-1}(T^{-\alpha})T^{\frac{\alpha}{\kappa_{0}}})^{\frac{(p-1)d}{2}}
		\int_{0}^{t_{1}}(t_{1}-s)^{(\alpha-\sigma_{2})p-\frac{\alpha d}{2\kappa_{0}}(p-1)}(\|\theta_{4}\|_{L_{1}}+\sup_{s\in[0,T]}\mathbb{E}\|w(s)\|_{L_{p}}^{p})\,ds<\infty.
	\end{align*}
	
	Therefore, by the dominated convergence theorem, we conclude that \eqref{continuous estimate} holds as $t_{2}\rightarrow t_{1}$.
\end{proof}
\begin{remark}
	In particular, for general L\'{e}vy time-space white noise $\mathcal{Z}_{t,x}$ and $p=2$, under the assumptions of Theorem \rm\ref{local mild solution theorem}, and if there exist $\theta_{4},\theta_{5}\in L_{2}(\mathbb{R}^{d})$ satisfying
	\[
	|h(t,x,z)|\lesssim(\theta_{4}(x)+|z|),\quad |h(t,x,z_{1})-h(t,x,z_{2})|\lesssim(\theta_{5}(x)+|z_{1}|+|z_{2}|)|z_{1}-z_{2}|,
	\]
	and if the $\mathcal{F}_{0}$-adapted process $w_{0}$ satisfies $\mathbb{E}[\|w_{0}\|_{L_{2}}^{2}]<\infty$, then NLSPDE \eqref{TSFSPDE1} admits a unique local mild solution $w$ on $[0,T]\times\mathbb{R}^{d}$ which has a predictable modification, and satisfies
	\[
	\mathbb{E}[\|w(t\wedge\upsilon,\cdot)\|_{2}^{2}]<\infty.
	\]
\end{remark}

Indeed, following the same proof procedure as in Theorem \ref{local mild solution theorem}, we define the mapping $\mathcal{T}w=\sum_{i=1}^{4}\mathcal{T}_{i}w$, where $\mathcal{T}_{1}w$, $\mathcal{T}_{2}w$, $\mathcal{T}_{3}w$ are defined as in Theorem \ref{local mild solution theorem} and
\[
\mathcal{T}_{4}w=\int_{0}^{t}\int_{\mathbb{R}^{d}}\mathcal{S}_{\alpha,\sigma_{1}}(t-s,z-y)h(s,y,w(s,y))W(dy,ds).
\]
Noting that the Gaussian white noise is isometric from $L_{2}(\mathbb{R}^{d})$ to the Gaussian space, the proof follows similarly to that of Theorem \ref{local mild solution theorem}, we omit it.

	\noindent{\bf Declaration of competing interest}\\
	The authors declare that they have no competing interests.\\
	\noindent{\bf Data availability}\\
	No data was used for the research described in the article.\\
	\noindent{\bf Acknowledgements}\\
	This work was supported by National Natural Science Foundation of China (12471172), Fundo para o Desenvolvimento das Ci\^{e}ncias e da Tecnologia of Macau (No. 0092/2022/A) and Hunan Province Doctoral Research Project CX20230633.

\end{document}